\documentclass{amsart}

\usepackage[english]{babel}
\usepackage{hyperref}
\usepackage{mathrsfs}
\usepackage{color}
\usepackage{tikz}
\usepackage{pgfplots}

\pgfplotsset{compat=1.16} 
\usepgfplotslibrary{fillbetween}

\numberwithin{equation}{section} 
\numberwithin{figure}{section} 

\theoremstyle{plain} 
\newtheorem{theorem}{Theorem}[section] 
\newtheorem{lemma}[theorem]{Lemma} 
\newtheorem{corollary}[theorem]{Corollary}

\theoremstyle{definition} 
\newtheorem*{definition}{Definition} 
\newtheorem{example}[theorem]{Example} 
\newtheorem{problem}[theorem]{Problem}

\DeclareMathOperator{\mre}{Re} 
\DeclareMathOperator{\ind}{ind}

\begin{document} 
	\title[Idempotent Fourier multipliers acting contractively on $H^p$ 
	spaces]{Idempotent Fourier multipliers \\ acting contractively on $H^p$ 
		spaces} 
	\date{\today} 
	
	\author{Ole Fredrik Brevig} 
	\address{Department of Mathematics, University of Oslo, 0851 Oslo, Norway} 
	\email{obrevig@math.uio.no}
	
	\author{Joaquim Ortega-Cerd\`{a}} 
	\address{Department de Matem\`{a}tiques i Inform\`{a}tica, Universitat de 
		Barcelona \& Barcelona Graduate school in mathematics, Gran Via 585, 
		08007 Barcelona, Spain} 
	\email{jortega@ub.edu}
	
	\author{Kristian Seip} 
	\address{Department of Mathematical Sciences, Norwegian University of 
		Science and Technology (NTNU), NO-7491 Trondheim, Norway} 
	\email{kristian.seip@ntnu.no} 
	\begin{abstract}
		We describe the idempotent Fourier multipliers that act contractively 
		on $H^p$ spaces of the $d$-dimensional torus $\mathbb{T}^d$ for $d\geq 
		1$ and $1\leq p \leq \infty$. When $p$ is not an even integer, such 
		multipliers are just restrictions of contractive idempotent multipliers 
		on $L^p$ spaces, which in turn can be described by suitably combining 
		results of Rudin and And\^{o}. When $p=2(n+1)$, with $n$ a positive 
		integer, contractivity depends in an interesting geometric way on $n$, 
		$d$, and the dimension of the set of frequencies associated with the 
		multiplier. Our results allow us to construct a linear operator that is 
		densely defined on $H^p(\mathbb{T}^\infty)$ for every $1 \leq p \leq 
		\infty$ and that extends to a bounded operator if and only if 
		$p=2,4,\ldots,2(n+1)$. 
	\end{abstract}
	
	
	\thanks{Ortega-Cerd\`{a} was partially supported by the Generalitat de 
		Catalunya (grant 2017 SGR 358) and the Spanish Ministerio de Ciencia, 
		Innovaci\'on y Universidades (project MTM2017-83499-P). Seip was 
		supported in part by the Research Council of Norway grant 275113}
	
	\maketitle

	\section{Introduction} This paper grew out of an attempt to clarify the 
	precise scope and nature of certain contractive inequalities that have 
	proven useful in the study of the Hardy spaces $H^p(\mathbb{T}^d)$ when 
	$d\geq 1$ and $1\leq p \leq \infty$. The inequalities in question can best 
	be seen as instances of idempotent Fourier multipliers that act 
	contractively on $H^p(\mathbb{T}^d)$, and our main purpose will therefore 
	be to describe such multipliers. 
	
	Since any Fourier multiplier on $L^p(\mathbb{T}^d)$ induces a Fourier 
	multiplier on $H^p(\mathbb{T}^d)$, it is natural to begin with the easier 
	problem of describing idempotent Fourier multipliers acting contractively 
	on $L^p(\mathbb{T}^d)$. To this end, we represent functions $f$ in 
	$L^p(\mathbb{T}^d)$ by their Fourier series $f(z) \sim \sum_{\alpha \in 
		\mathbb{Z}^d} \widehat{f}(\alpha)\,z^\alpha$, where
	\[\widehat{f}(\alpha) := \int_{\mathbb{T}^d} f(z) 
	\,\overline{z^\alpha}\,dm_d(z)\]
	and $m_d$ denotes the Haar measure of the $d$-dimensional torus 
	$\mathbb{T}^d$. For $\Lambda$ a non-empty subset of $\mathbb{Z}^d$, we 
	consider the operator $P_{\Lambda}$ that is densely defined on 
	$L^p(\mathbb{T}^d)$ by the rule
	\[P_\Lambda f(z) := \sum_{\alpha \in \Lambda} \widehat{f}(\alpha) 
	z^\alpha.\]
	The operator $P_\Lambda$ is an idempotent Fourier multiplier, since it 
	corresponds to pointwise multiplication of the Fourier coefficients 
	$\widehat{f}(\alpha)$ by the characteristic function of $\Lambda$. We will 
	say that $\Lambda$ is a \emph{contractive projection set for} 
	$L^p(\mathbb{T}^d)$ when $P_\Lambda$ extends to a contraction on 
	$L^p(\mathbb{T}^d)$. Following Rudin \cite{Rudin}, we say that a subset 
	$\Lambda$ of $\mathbb{Z}^d$ is a \emph{coset} in $\mathbb{Z}^d$ if 
	$\Lambda$ is equal to the coset of a subgroup of $(\mathbb{Z}^d,+)$. The 
	following result can be deduced by suitably combining arguments and results 
	due to Rudin \cite{Rudin} and And\^{o} \cite{Ando66}. Note that the case 
	$p=2$ is omitted in the statement, since every non-empty subset of 
	$\mathbb{Z}^d$ is trivially a contractive projection set for 
	$L^2(\mathbb{T}^d)$.
	\begin{theorem}\label{thm:Tdproj} 
		Let $d$ be a non-negative integer and fix $1 \leq p \leq \infty$, 
		$p\neq 2$. A subset $\Lambda$ of $\mathbb{Z}^d$ is a contractive 
		projection set for $L^p(\mathbb{T}^d)$ if and only if $\Lambda$ is a 
		coset in $\mathbb{Z}^d$. 
	\end{theorem}
	
	Theorem~\ref{thm:Tdproj} has a striking bearing on the question of when 
	$P_\Lambda$ extends to a bounded operator on $L^1(\mathbb T^d)$. Indeed, 
	results of Helson \cite{Helson53} in dimension $1$ and Rudin \cite{Rudin59} 
	in higher dimensions show that $P_\Lambda$ defines a bounded linear 
	operator on $L^1(\mathbb{T}^d)$ if and only if $\Lambda = \bigcup_{k=1}^n 
	\Lambda_k$, where $\Lambda_1,\ldots,\Lambda_n$ are cosets of 
	$\mathbb{Z}^d$. By a celebrated paper of Cohen \cite{Cohen60}, this result 
	extends to $L^1(G)$ for $G$ a compact abelian group. It remains however a 
	difficult open problem to describe the sets $\Lambda$ that yield bounded 
	operators $P_{\Lambda}$ on $L^p(\mathbb{T}^d)$ when $p\neq 1,2$.
	
	We mention two examples of frequently encountered inequalities that are 
	covered by Theorem~\ref{thm:Tdproj}. The first of these is an inequality of 
	F.~Wiener that appeared already in Bohr's classical work on what later 
	became known as the Bohr radius \cite{Bohr14}. In our terminology, this is 
	just the case $d=1$ of Theorem~\ref{thm:Tdproj}. See 
	\cite[Sec.~1.7]{MSUV15} for a recent function theoretic application and 
	\cite{BoasKhavinson97} for a $d$-dimensional version of it. The second 
	example inequality deals with the restriction to the $m$-homogeneous terms 
	of a power series in $d$ variables. This is again a special case of 
	Theorem~\ref{thm:Tdproj}, with the dimension of the coset being strictly 
	smaller than the dimension of the ambient space $\mathbb{Z}^d$. We refer to 
	\cite{BQS16}, \cite{BS16} and \cite[Sec. 9]{ColeGamelin86} for respectively 
	an operator, number, and function theoretic application of the
	corresponding contractive inequality.
	
	Our main theorem shows that there are contractive projection sets for 
	$H^p(\mathbb{T}^d)$ that are not covered by Theorem~\ref{thm:Tdproj} when 
	$p$ is an even integer $\geq 4$. To state this result, we recall first that 
	$H^p(\mathbb{T}^d)$ is the subspace of $L^p(\mathbb{T}^d)$ comprised of 
	functions $f$ such that $\widehat{f}(\alpha)=0$ for every $\alpha$ in 
	$\mathbb{Z}^d \setminus \mathbb{N}_0^d$, where $\mathbb{N}_0 := 
	\{0,1,2,\ldots\}$. We will say that a subset $\Gamma$ of $\mathbb{N}_0^d$ 
	is a \emph{contractive projection set for} $H^p(\mathbb{T}^d)$ if 
	$P_\Gamma$ extends to a contraction on $H^p(\mathbb{T}^d)$. Since 
	$H^p(\mathbb{T}^d)$ is a subspace of $L^p(\mathbb{T}^d)$, we get the 
	following immediate consequence of Theorem~\ref{thm:Tdproj}. If $\Lambda$ 
	is a coset in $\mathbb{Z}^d$, then $\Lambda \cap \mathbb{N}_0^d$ is a 
	contractive projection set for $H^p(\mathbb{T}^d)$. We are interested in 
	knowing if there are other contractive projection sets for $H^p(\mathbb 
	T^d)$. It turns out that the dimension of the affine span of $\Gamma$, 
	henceforth called $\dim(\Gamma)$ or the \emph{dimension of} $\Gamma$, 
	plays a nontrivial role in this problem, and we therefore make the 
	following definition. 
	\begin{definition}
		Suppose that $1\leq k \leq d$. We say that $H^p(\mathbb{T}^d)$ enjoys 
		the \emph{contractive restriction property of dimension $k$} if every 
		$k$-dimensional contractive projection set for $H^p(\mathbb{T}^d)$ is 
		of the form $\Lambda \cap \mathbb{N}_0^d$ with $\Lambda$ a coset in 
		$\mathbb Z^d$. 
	\end{definition}
	
	Now our main result reads as follows. 
	\begin{theorem}\label{thm:TdprojHp} 
		Suppose that $1\leq p \leq \infty$. 
		\begin{enumerate}
			\item[(a)] If $d=2$ or $k=1$, then $H^p(\mathbb{T}^d)$ enjoys the 
			contractive restriction property of dimension $k$ if and only if 
			$p\neq 2$. 
			\item[(b)] If either $d=k=3$ or $d\geq 3$ and $k=2$, then 
			$H^p(\mathbb{T}^d)$ enjoys the contractive restriction property of 
			dimension $k$ if and only if $p\neq 2, 4$. 
			\item[(c)] If $d\geq 4$ and $k\geq 3$, then $H^p(\mathbb{T}^d)$ 
			enjoys the contractive restriction property of dimension $k$ if and 
			only if $p$ is not an even integer. 
		\end{enumerate}
	\end{theorem}
	
	One may think suggestively of the case $d\geq 4$ and $k\geq 3$ as 
	exhibiting higher-dimensional behavior. We will see that the hardest part 
	of the theorem is item (b) which can be thought of as representing the two 
	cases of intermediate dimension, namely $d=k=3$ and $d\geq 3$, $k=2$.
	
	As regards the variation in $p$, the simplest part of the proof of 
	Theorem~\ref{thm:TdprojHp} is the case $p=\infty$, because we can construct 
	explicit examples demonstrating that any contractive projection set must be 
	the restriction of a coset to $\mathbb{N}_0^d$. This is made possible by 
	the fact that the norm of $H^\infty(\mathbb{T}^d)$ is easy to understand. 
	In the case $1 \leq p < \infty$, we will by contrast reformulate the 
	problem using duality arguments (see e.g.~Shapiro's monograph 
	\cite[Sec.~4.2]{Shapiro}). In this approach, it is crucial to understand 
	the Fourier coefficients of
	\[|f|^{p-2} f\]
	in terms of the Fourier coefficients of $f$. It is clear that this problem 
	takes on a completely different character when $p$ is an even integer, in 
	which case we have an interesting geometric description of the contractive 
	projection sets that depend crucially on $p$.
	
	Suppose that $\Gamma$ is a non-empty subset of $\mathbb{N}_0^d$, and let 
	$\Lambda(\Gamma)$ denote the coset in $\mathbb{Z}^d$ generated by $\Gamma$. 
	We can represent every $\lambda$ in $\Lambda(\Gamma)$ as a finite linear 
	combination 
	\begin{equation}\label{eq:lgrep} 
		\lambda = \gamma + \sum_{\substack{\alpha \in \Gamma \\
				\alpha \neq \gamma}} m_{\gamma,\alpha} (\alpha-\gamma) , 
	\end{equation}
	where $\gamma$ is any element in $\Gamma$ and $m_{\gamma,\alpha}$ are 
	integers.
	\begin{definition}
		Let $\Gamma$ be a non-empty subset of $\mathbb{N}_0^d$ and suppose that 
		$\lambda$ is in $\Lambda(\Gamma)$. The \emph{distance} from $\Gamma$ to 
		$\lambda$ is
		\[d(\Gamma,\lambda) := \inf\max\left(\sum_{m_{\gamma,\alpha}>0} 
		m_{\gamma,\alpha}, -\sum_{m_{\gamma,\alpha}<0} 
		m_{\gamma,\alpha}\right)\]
		where the infimum is taken over all possible representations 
		\eqref{eq:lgrep} of $\lambda$. For a non-negative integer $n$, the 
		$n$\emph{-extension} of $\Gamma$ is
		\[E_n(\Gamma) := \{\lambda \in \Lambda(\Gamma)\cap 
		\mathbb{N}_0^d\,:\,d(\Gamma,\lambda) \leq n \}.\]
	\end{definition}
	
	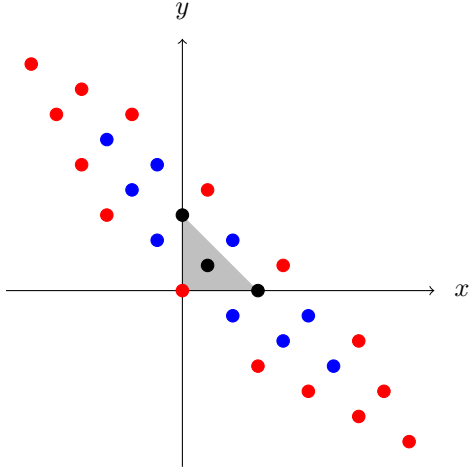
\begin{figure}
		\centering
		\begin{tikzpicture}
			\begin{axis}[
				axis equal image,
				xlabel= $x$, 
				ylabel = $y$,
				axis lines = center,
				xmin = -7,
				xmax = 10,
				ymin = -7,
				ymax = 10,
				every axis x label/.style={
					at={(ticklabel* cs:1.025)},
					anchor=west,
				},
				every axis y label/.style={
					at={(ticklabel* cs:1.025)},
					anchor=south,
				},
				ticks=none,%
				axis line style={->}]
				
				\addplot[thin, name path=t1, draw opacity=0, domain=0:3, 
				samples=5] {3-x};
				\addplot[thin, name path=t2, draw opacity=0, domain=0:3, 
				samples=5] {0};
				\addplot[color=black, fill=black, fill opacity=0.25]
				fill between[of=t1 and t2, soft clip={domain=0:3},];
				
				\node[circle, draw, fill=black, scale=0.5] at (3,0){};
				\node[circle, draw, fill=black, scale=0.5] at (0,3){};
				\node[circle, draw, fill=black, scale=0.5] at (1,1){};
				
				\node[circle, draw, color=red, fill=red, scale=0.5] at (0,0){};
				
				\node[circle, draw, color=blue, fill=blue, scale=0.5, 
				opacity=1] 
				at (-1,2){};
				\node[circle, draw, color=blue, fill=blue, scale=0.5, 
				opacity=1] 
				at (2,-1){};
				\node[circle, draw, color=blue, fill=blue, scale=0.5, 
				opacity=1] 
				at (2,2){};
				\node[circle, draw, color=blue, fill=blue, scale=0.5, 
				opacity=1] 
				at (4,-2){};
				\node[circle, draw, color=blue, fill=blue, scale=0.5, 
				opacity=1] 
				at (-2,4){};
				\node[circle, draw, color=blue, fill=blue, scale=0.5, 
				opacity=1] 
				at (5,-1){};
				\node[circle, draw, color=blue, fill=blue, scale=0.5, 
				opacity=1] 
				at (-1,5){};
				\node[circle, draw, color=blue, fill=blue, scale=0.5, 
				opacity=1] 
				at (6,-3){};
				\node[circle, draw, color=blue, fill=blue, scale=0.5, 
				opacity=1] 
				at (-3,6){};
				
				\node[circle, draw, color=red, fill=red, scale=0.5, opacity=1] 
				at (-3,3){};
				\node[circle, draw, color=red, fill=red, scale=0.5, opacity=1] 
				at (3,-3){};
				
				\node[circle, draw, color=red, fill=red, scale=0.5, opacity=1] 
				at (-4,5){};
				\node[circle, draw, color=red, fill=red, scale=0.5, opacity=1] 
				at (5,-4){};
				
				\node[circle, draw, color=red, fill=red, scale=0.5, opacity=1] 
				at (7,-5){};
				\node[circle, draw, color=red, fill=red, scale=0.5, opacity=1] 
				at (-5,7){};
				
				\node[circle, draw, color=red, fill=red, scale=0.5, opacity=1] 
				at (7,-2){};
				\node[circle, draw, color=red, fill=red, scale=0.5, opacity=1] 
				at (-2,7){};
				
				\node[circle, draw, color=red, fill=red, scale=0.5, opacity=1] 
				at (8,-4){};
				\node[circle, draw, color=red, fill=red, scale=0.5, opacity=1] 
				at (-4,8){};
				
				\node[circle, draw, color=red, fill=red, scale=0.5, opacity=1] 
				at (9,-6){};
				\node[circle, draw, color=red, fill=red, scale=0.5, opacity=1] 
				at (-6,9){};
				
				\node[circle, draw, color=red, fill=red, scale=0.5, opacity=1] 
				at (4,1){};
				\node[circle, draw, color=red, fill=red, scale=0.5, opacity=1] 
				at (1,4){};
			\end{axis} 
		\end{tikzpicture}
		\caption{Points $\lambda$ which satisfy 
			{\color{blue}$d(\Gamma,\lambda)=1$} and 
			{\color{red}$d(\Gamma,\lambda)=2$} for 
			$\Gamma=\{(3,0,0),(0,3,0),(1,1,1)\}$, represented in the projected 
			plane defined by $z=3-x-y$. The shaded triangle represents the 
			intersection of this plane and the narrow cone. Note that $(0,0,3)$ 
			is in $E_2(\Gamma)$, so $E_2(\Gamma)=\Lambda(\Gamma)\cap 
			\mathbb{N}_0^3$. However $(0,0,3)$ is not in $E_1(\Gamma)$, so 
			$E_1(\Gamma)=\Gamma$.}
		\label{fig:nocorner}
	\end{figure}
	
	Clearly, $\Lambda(E_n(\Gamma))=\Lambda(\Gamma)$ for every $n\geq1$. 
	Moreover, we find that $\Gamma = \Lambda \cap \mathbb{N}_0^d$ for a coset 
	$\Lambda$ in $\mathbb{Z}^d$ if and only if
	\[\Gamma = \bigcup_{n=1}^\infty E_n(\Gamma).\]
	See Figure~\ref{fig:nocorner} for an example illustrating the possibility 
	that $E_2(\Gamma) \neq E_1(\Gamma)=\Gamma$.
	
	The proof of Theorem~\ref{thm:TdprojHp} in the case that $p$ is an even 
	integer, which is the most difficult case, is based on the following result.
	\begin{theorem}\label{thm:shapirogeometric} 
		Let $d$ be a positive integer and $n$ be a non-negative integer. A set 
		$\Gamma$ in $\mathbb{N}_0^d$ is a contractive projection set for 
		$H^{2(n+1)}(\mathbb{T}^d)$ if and only if $E_n(\Gamma)=\Gamma$. 
	\end{theorem}
	
	Theorem~\ref{thm:shapirogeometric} gives rise to an effective algorithm for 
	checking whether a finite subset $\Gamma$ of $\mathbb{N}_0^d$ is a 
	contractive projection set for $H^{2(n+1)}(\mathbb{T}^d)$.
	
	The $d$ and $k$ dependence of Theorem~\ref{thm:TdprojHp} appears when we 
	operationalize the condition of Theorem~\ref{thm:shapirogeometric}. 
	Inspired by a suggestive terminology introduced by Helson \cite{Helson06}, 
	we will sometimes refer to $\mathbb{N}_0^d$ as the \emph{narrow cone} in 
	$\mathbb{Z}^d$ to visualize how the geometry changes when $d$ increases: 
	$\mathbb{N}_0^d$ becomes narrower in $\mathbb{Z}^d$, and this permits more 
	sets $\Gamma$ to enjoy the crucial property that $E_n(\Gamma)=\Gamma$. 
	
	Two of our examples reflecting the kind of narrowness just alluded to, has 
	an interesting application in the limiting case $d=\infty$. To state this 
	final result of the present paper, we first define $\mathbb{T}^\infty$ as 
	the countably infinite product of the torus $\mathbb{T}$ and equip it with 
	its Haar measure $m_\infty$. The dual group of $\mathbb{T}^\infty$ is
	\[\mathbb{Z}^{(\infty)} = \bigcup_{d=1}^\infty \mathbb{Z}^d\]
	in view of the natural inclusion $\mathbb{Z}^d \subseteq \mathbb{Z}^{d+1}$. 
	Fix $1 \leq p \leq \infty$. Every $f$ in $L^p(\mathbb{T}^\infty)$ can be 
	represented as a Fourier series $f(z) \sim \sum_{\alpha \in 
		\mathbb{Z}^{(\infty)}} \widehat{f}(\alpha)\,z^\alpha$, where
	\[\widehat{f}(\alpha) = \int_{\mathbb{T}^\infty} 
	f(z)\,\overline{z^\alpha}\,dm_\infty(z).\]
	The Hardy space $H^p(\mathbb{T}^\infty)$ is the subspace of 
	$L^p(\mathbb{T}^\infty)$ comprised of functions $f$ such that 
	$\widehat{f}(\alpha)=0$ for every $\alpha$ in $\mathbb{Z}^{(\infty)} 
	\setminus \mathbb{N}_0^{(\infty)}$. It is not hard to see that 
	Theorem~\ref{thm:Tdproj}, Theorem~\ref{thm:TdprojHp}, and 
	Theorem~\ref{thm:shapirogeometric} extend to the infinite-dimensional 
	torus. 
	
	Bayart and Masty{\l}o \cite{BM19} have recently demonstrated that there are 
	no variants of the classical real and complex interpolation theorems for 
	$H^p(\mathbb{T}^\infty)$ in contrast to the finite dimensional case. The 
	following result strikingly exemplifies the impossibility of interpolating 
	between Hardy spaces on the infinite-dimensional torus.
	\begin{theorem}\label{thm:curlyop} 
		Fix an integer $n\geq1$. There is a linear operator $T_n$ which is 
		densely defined on $H^p(\mathbb{T}^\infty)$ for every $1 \leq p \leq 
		\infty$, and which does not extend to a bounded operator on 
		$H^p(\mathbb{T}^\infty)$ unless $p=2,4,\ldots,2(n+1)$. 
	\end{theorem}
	
	Our main interest in Theorem~\ref{thm:curlyop} stems from the \emph{Bohr 
		correspondence}, which allows us to translate results from Hardy spaces 
	on the infinite-dimensional torus to Hardy spaces of Dirichlet series. 
	Readers familiar with that field of study will immediately notice the 
	partial analogy between Theorem~\ref{thm:curlyop} and the local embedding 
	problem (see \cite[Sec.~3]{SS09} or Section~\ref{sec:bohr} below). However, 
	it should be stressed that the construction of Theorem~\ref{thm:curlyop} is 
	purely multiplicative, while the local embedding problem concerns the 
	interplay between the additive and multiplicative structure of the integers.
	
	We close this introduction by giving a brief overview of the contents of 
	the three additional sections of this paper. Section~\ref{sec:Tdproj} 
	contains an exposition of the proof of Theorem~\ref{thm:Tdproj} and the 
	proof of Theorem~\ref{thm:TdprojHp} in the case $p=\infty$. The body of the 
	paper is Section~\ref{sec:TdprojHp} which deals with contractive projection 
	sets for $H^p(\mathbb{T}^d)$ and contains the proof of 
	Theorem~\ref{thm:shapirogeometric} and Theorem~\ref{thm:TdprojHp} for 
	$p<\infty$. In the final Section~\ref{sec:bohr}, we establish 
	Theorem~\ref{thm:curlyop} and discuss our results in the context of Hardy 
	spaces of Dirichlet series.

	\section{Contractive projection sets for 
		\texorpdfstring{$L^p(\mathbb{T}^d)$}{Lp(Td)}} \label{sec:Tdproj}
	
	\subsection{Proof of Theorem~\ref{thm:Tdproj}} \label{subsec:Tdproof} The 
	purpose of this section is to present a self-contained proof of 
	Theorem~\ref{thm:Tdproj}. As mentioned above, this can be achieved by 
	combining arguments and results due to Rudin \cite{Rudin} and And\^{o} 
	\cite{Ando66}. For expositional reasons, we have nevertheless chosen to 
	furnish a complete proof.
	
	Suppose that $\Lambda$ is a non-empty subset of $\mathbb{Z}^d$. We begin by 
	noting that we may assume without loss of generality that $0$ is in 
	$\Lambda$. Indeed, suppose that this is not the case. Fix some $\lambda$ in 
	$\Lambda$ and consider the translated set
	\[\Lambda-\lambda := \left\{\alpha \in \mathbb{Z}^d \,:\, \alpha+\lambda 
	\in \Lambda\right\}\]
	which clearly contains $0$. Define $M_\lambda$ on $L^p(\mathbb{T}^d)$ by 
	$M_\lambda f(z) := z^\lambda f(z)$. Evidently, $M_\lambda$ is an isometric 
	isomorphism on $L^p(\mathbb{T}^d)$. Note also that 
	\begin{equation}\label{eq:Mfact} 
		P_{\Lambda} M_\lambda = M_\lambda P_{\Lambda-\lambda}, 
	\end{equation}
	which implies at once that $\Lambda$ is a contractive projection set for 
	$L^p(\mathbb{T}^d)$ if and only if $\Lambda-\lambda$ is a contractive 
	projection set for $L^p(\mathbb{T}^d)$.
	
	The following part of the proof is from Section~3.1.2 and Section~3.2.4 in 
	Rudin's monograph \cite{Rudin}, where the analogous statement is 
	established for $L^1(G)$ with $G$ a compact abelian group. This part of 
	Rudin's argument extends to $1< p \leq \infty$, $p\neq2$, without 
	modification.
	\begin{proof}
		[Proof of Theorem~\ref{thm:Tdproj}: Sufficiency] As noted above, we may 
		restrict our attention to the case that $\Lambda$ is a subgroup of 
		$\mathbb{Z}^d$ by translating a coset if necessary.
		
		We will require some preliminary results regarding $\mathbb{T}^d$ and 
		$\mathbb{Z}^d$. Recall that $\mathbb{T}^d$ is a compact abelian group 
		whose dual group is $\mathbb{Z}^d$. Suppose that $\Lambda$ is a 
		subgroup of $\mathbb{Z}^d$. The \emph{annihilator}
		\[\Lambda^\perp := \left\{z \in \mathbb{T}^d \,:\, z^\alpha=1 \text{ 
			for every } \alpha \text{ in } \Lambda\right\}\]
		is the dual group of the coset group $\mathbb{Z}^d/\Lambda$ (see 
		e.g.~\cite[Thm.~2.1.2]{Rudin}). Since $\Lambda^\perp$ is a closed 
		subgroup of $\mathbb{T}^d$, it is a compact abelian group whose Haar 
		measure we shall denote by $\mu$. By the duality relations between 
		$\mathbb{Z}^d/\Lambda$ and $\Lambda^\perp$, we may represent the 
		characteristic function of $\Lambda$ in $\mathbb{Z}^d$ as
		\[\mathbf{1}_\Lambda(\alpha) = \int_{\Lambda^\perp} 
		\zeta^\alpha\,d\mu(\zeta).\]
		For any $f$ in $L^p(\mathbb{T}^d)$, set $f_\zeta(z) := f(\zeta_1 z_1, 
		\zeta_2 z_2, \ldots, \zeta_d z_d)$. We now find that 
		\begin{equation}\label{eq:annirep} 
			P_\Lambda f(z) = \sum_{\alpha \in \Lambda} \widehat{f}(\alpha)\, 
			z^\alpha = \sum_{\alpha \in \mathbb{Z}^d} 
			\mathbf{1}_\Lambda(\alpha) \widehat{f}(\alpha) \, z^\alpha = 
			\int_{\Lambda^\perp} f_\zeta(z)\,d\mu(\zeta). 
		\end{equation}
		Taking the $L^p(\mathbb{T}^d)$ norm on both sides and combining 
		Minkowski's inequality and the fact that $\|f_\zeta\|_p=\|f\|_p$ for 
		every $\zeta$ in $\mathbb{T}^d$, we see that
		\[ \|P_\Lambda f \|_p \leq \|f\|_p 
		\left(\mu(\Lambda^\perp)\right)^\frac{1}{p}= \|f\|_p, \]
		since $\mu$ is the Haar measure of $\Lambda^\perp$. 
	\end{proof}
	
	Observe that \eqref{eq:annirep} says that any projection with respect to a 
	subgroup $\Lambda$ is in fact an averaging operator in $L^1(\mathbb{T}^d)$. 
	Douglas \cite{Douglas65} proved that any projection in $L^1$ that fixes the 
	constants, is in fact a conditional expectation with respect to a 
	sigma-algebra. 
	
	Before we proceed with the proof that every contractive projection set in 
	$L^p(\mathbb{T}^d)$ for $1 \leq p \leq \infty$, $p\neq 2$, is necessarily a 
	coset in $\mathbb{Z}^d$, let us explain how F.~Wiener's projection and the 
	$m$-homogeneous projection mentioned in the introduction fit into the 
	framework of the proof presented above. Note that by 
	Theorem~\ref{thm:Tdproj}, we see that Example~\ref{ex:fwiener} and 
	Example~\ref{ex:mhom} contain all the contractive projection sets for 
	$L^p(\mathbb{T})$, since the only cosets in $\mathbb{Z}$ are the arithmetic 
	progressions and the singletons.
	\begin{example}\label{ex:fwiener} 
		The contractive projection sets that correspond to F.~Wiener's 
		projection, are the arithmetic progressions $\Lambda := r +k\mathbb{Z}$ 
		for integers $k>1$ and $0 \leq r < k$. The associated subgroup of 
		$\mathbb{Z}$ is $k\mathbb{Z}$ and clearly
		\[(k\mathbb{Z})^\perp = \big\{\omega_k^j \,:\, 
		j=0,1,\ldots,{k-1}\big\},\]
		where $\omega_k = \exp(2\pi i /k)$. The Haar measure of 
		$(k\mathbb{Z})^\perp$ is the normalized counting measure. Combining 
		\eqref{eq:Mfact} and \eqref{eq:annirep}, we get the well-known formula
		\[P_\Lambda f(z) = \frac{1}{k}\sum_{j=0}^{k-1} f(\omega_k^jz) 
		\,\omega_k^{-jr}\]
		for $f$ in $L^p(\mathbb{T})$. 
	\end{example}
	\begin{example}\label{ex:mhom} 
		Fix $d\geq1$. For an integer $m$, the $m$-homogeneous projection on 
		$L^p(\mathbb{T}^d)$ corresponds to the contractive projection set
		\[\Lambda_m := \big\{\alpha \in \mathbb{Z}^d \,:\, 
		\alpha_1+\alpha_2+\cdots + \alpha_d = m\big\}.\]
		The associated subgroup of $\mathbb{Z}^d$ is $\Lambda_0$, so 
		$\Lambda_0^\perp=\big\{z=(w,\ldots, w) \,:\, w\in \mathbb{T}\big\}$, 
		and the Haar measure is $m_1$ on $\mathbb T$. Combining 
		\eqref{eq:Mfact} and \eqref{eq:annirep}, we get
		\[P_{\Lambda_m} f(z) = \int_{\mathbb{T}} f(z_1w,z_2w,\ldots,z_d w) 
		\,w^{-m}\,dm_1(w)\]
		for $f$ in $L^p(\mathbb{T}^d)$. 
	\end{example}
	
	By results in Section~1.4.1 and Section~3.2.3 in Rudin's monograph 
	\cite{Rudin}, it follows that if $\Gamma$ is a contractive projection set 
	for $L^1(\mathbb{T}^d)$, then $\Gamma$ is necessarily a coset in 
	$\mathbb{Z}^d$. This part of Rudin's argument does not work for $p>1$. 
	However, we can appeal to a general result of 
	And\^{o}~\cite[Thm.~1]{Ando66} which states that any contractive projection 
	on $L^p$ for $1<p<\infty$, $p\neq2$, which fixes the constants, extends to 
	a contractive projection on $L^1$. Hence any contractive projection set for 
	$L^p(\mathbb{T}^d)$, for $1<p<\infty$, $p\neq2$, must be a coset in 
	$\mathbb{Z}^d$. The case $p=\infty$ is handled by Riesz--Thorin 
	interpolation, since the linear operator $P_\Lambda$ is contractive on 
	$L^p$ for $2<p<\infty$ when it is contractive on $L^2$ and $L^\infty$. 
	These considerations also apply if $\mathbb{T}^d$ is replaced by a compact 
	abelian group $G$.
	
	To highlight the new difficulties that arise when we later treat the 
	corresponding problem for $H^p(\mathbb{T}^d)$, we will present a direct 
	proof of the necessity part of Theorem~\ref{thm:Tdproj} below. We shall 
	require two preliminary estimates. We note in passing that it is possible 
	to obtain similar estimates if $\mathbb{T}^d$ is replaced by a compact 
	abelian group $G$, thereby sidestepping the need for And\^{o}'s theorem and 
	Riesz--Thorin interpolation.
	\begin{lemma}\label{lem:Tdline} 
		Fix $1 \leq p \leq \infty$, $p\neq 2$, and set $c_p := 2/p-1$. Then 
		\begin{equation}\label{eq:Tdline} 
			\|c_p \varepsilon \overline{z}+1+\varepsilon z \|_p < 
			\|1+\varepsilon z \|_p 
		\end{equation}
		for every sufficiently small $\varepsilon>0$. 
	\end{lemma}
	\begin{proof}
		Let $c$ be a real number and compute 
		\begin{equation}\label{eq:lineabs2} 
			|c \varepsilon \overline{z}+1+\varepsilon z|^2 = 1 + 
			\varepsilon(1+c) (z+\overline{z}) + \varepsilon^2 \left((1+c^2) + 
			c(z^2+{\overline{z}^2}) \right). 
		\end{equation}
		Using the binomial expansion, we find that 
		\begin{align*}
			|c \varepsilon \overline{z}+1+\varepsilon z|^p = 1 &+ \varepsilon 
			\,\frac{p}{2}(1+c) (z+\overline{z}) \\
			&+ \varepsilon^2 \left(\frac{p}{2}\left((1+c^2) + 
			c(z^2+{\overline{z}^2})\right) + \binom{p/2}{2}(1+c)^2 
			(z+\overline{z})^2\right)\\
			&+O(\varepsilon^3) 
		\end{align*}
		for every sufficiently small $\varepsilon>0$. Integrating over 
		$\mathbb{T}$ and simplifying, we get
		\[\|c \varepsilon \overline{z}+1+\varepsilon z\|_p^p = 1 + 
\left(\frac{p^2}{4}(c+1-2/p)^2+p-1 
		\right)\varepsilon^2 + O(\varepsilon^3).\]
		If $1 \leq p < \infty$ and $\varepsilon>0$ is sufficiently small, then 
		the minimum is attained at $c=2/p-1$, which yields \eqref{eq:Tdline}. 
		It remains to deal with the case $p=\infty$. Inspecting 
		\eqref{eq:lineabs2} with $c=c_\infty=-1$, we find that the supremum is 
		attained at $z = \pm i$. Consequently, \eqref{eq:Tdline} reduces in this 
		case to
		\[\sqrt{1+4\varepsilon^2}<1+\varepsilon ,\]
		which holds for all sufficiently small $\varepsilon>0$. 
	\end{proof}
	\begin{lemma}\label{lem:Tdplane} 
		Fix $1 \leq p < \infty$, $p\neq 2$, and set $c_p := 1-p/2$. Then 
		\begin{align}
			\|1+ \varepsilon (z_1+z_2) + c_p \,\varepsilon^2 z_1 z_2\|_p &< 
			\|1+ \varepsilon (z_1+z_2)\|_p \label{eq:Tdplane1} \intertext{for 
				every sufficiently small $\varepsilon>0$. Moreover,} 
			\|1+z_1+z_2-z_1z_2\|_\infty &< \|1+z_1+z_2\|_\infty. 
			\label{eq:Tdplane2} \end{align}
	\end{lemma}
	\begin{proof}
		Let $c$ be a fixed real number. For sufficiently small $\varepsilon>0$, 
		expand 
		\begin{equation}\label{eq:Tdplanebinomial} 
			\big(1+ \varepsilon (z_1+z_2) + c \,\varepsilon^2 z_1 
			z_2\big)^{p/2} = \sum_{j=0}^\infty \binom{p/2}{j} \big(\varepsilon 
			(z_1+z_2) + c \,\varepsilon^2 z_1 z_2\big)^j. 
		\end{equation}
		In this expansion, any monomial of degree $m$ will have $\varepsilon^m$ 
		in front of it. Hence we can rearrange in terms of $m$-homogeneous 
		polynomials to obtain 
		\begin{equation}\label{eq:mhomexp} 
			\big(1+\varepsilon(z_1+z_2)+c\varepsilon^2z_1 z_2\big)^{p/2} = 
			\sum_{m=0}^\infty \varepsilon^m P_m(z). 
		\end{equation}
		Here $P_m$ is an $m$-homogeneous polynomial whose coefficients do not 
		depend on $\varepsilon$. Since $P_m \perp P_n$ in $L^2(\mathbb{T}^2)$ 
		for $m \neq n$, we get from \eqref{eq:mhomexp} that 
		\begin{equation}\label{eq:mhomnorm} 
			\|1+\varepsilon(z_1+z_2)+c \varepsilon^2z_1 z_2\|_p^p = 
			\sum_{m=0}^\infty \varepsilon^{2m} \|P_m\|_2^2. 
		\end{equation}
		We need the first three terms, which we can read off from 
		\eqref{eq:Tdplanebinomial}. They are 
		\begin{align*}
			P_0(z) &= 1, \\
			P_1(z) &= \frac{p}{2}(z_1+z_2), \\
			P_2(z) &= \frac{p}{2}(c+p/2-1)z_1z_2 + \binom{p/2}{2}(z_1^2+z_2^2). 
		\end{align*}
		Inserting this into \eqref{eq:mhomnorm} we find that 
		\begin{align*}
			\|1+\varepsilon(z_1+z_2)+c \varepsilon^2z_1 z_2\|_p^p = 1 &+ 
			\frac{p^2}{2} \varepsilon^2 \\
			&+ \frac{p^2}{4} \left((c+p/2-1)^2 + \frac{(p-2)^2}{8}\right) 
			\varepsilon^4 +O(\varepsilon^6). 
		\end{align*}
		Hence if $1 \leq p < \infty$, $p\neq 2$, and $\varepsilon>0$ is 
		sufficiently small, then the minimum is attained at $c=1-p/2$, and 
		\eqref{eq:Tdplane1} follows. 
		
		It remains to establish \eqref{eq:Tdplane2}. The right-hand side is 
		clearly equal to $3$. For the left-hand side, we rewrite 
		$1+z_1+z_2-z_1z_2 = 1+z_2 + z_1(1-z_2)$ which implies that
		\[\|1+z_1+z_2-z_1z_2\|_\infty = \sup_{z_2 \in \mathbb{T}} 
		\left(|1+z_2|+|1-z_2|\right) = 2\sqrt{2}.\]
		Hence \eqref{eq:Tdplane2} holds since the right-hand side equals $3$. 
	\end{proof}
	\begin{proof}
		[Proof Theorem~\ref{thm:Tdproj}: Necessity] Fix $1 \leq p < \infty$, 
		$p\neq 2$, and suppose that $\Lambda$ is a contractive projection set 
		for $L^p(\mathbb{T}^d)$. As above, we may assume without loss of 
		generality that $0$ is in $\Lambda$, and we are therefore required to 
		prove that $\Lambda$ is a subgroup of $\mathbb{Z}^d$. If 
		$\Lambda=\{0\}$, there is nothing to prove, so we shall assume that 
		there is at least one element $\neq 0$ in $\Lambda$ and use this to 
		establish that $\Lambda$ must be closed under the group operations. 
		
		Suppose that $\alpha$ is in $\Lambda \setminus\{0\}$. By substituting 
		$z = z^\alpha$ in Lemma~\ref{lem:Tdline} and using that $P_\Lambda$ is 
		a contraction on $L^p(\mathbb{T}^d)$, we conclude at once that 
		$-\alpha$ must be in $\Lambda$. Suppose next that $\alpha$ and $\beta$ 
		are two (not necessarily distinct) elements in $\Lambda \setminus\{0\}$. 
		We need to show that $\alpha+\beta$ is in $\Lambda$. There are two 
		cases.
		
		If $j \alpha \neq k \beta$ for every pair of integers $j,k\neq0$, then 
		we may substitute $z_1 = z^\alpha$ and $z_2 = z^\beta$ in 
		Lemma~\ref{lem:Tdplane}. The fact that $P_\Lambda$ is a contraction on 
		$L^p(\mathbb{T}^d)$ implies at once that $\alpha+\beta$ is in 
		$\Lambda$, since $z_1 z_2 = z^{\alpha+\beta}$.
		
		If $j \alpha = k \beta$ for integers $j,k\neq0$, then we may write 
		$\alpha = a \gamma$ and $\beta = b\gamma$, where $a,b$ are integers and 
		$\gamma$ in $\mathbb{Z}^d$ satisfies 
		$\gcd(\gamma_1,\gamma_2,\ldots,\gamma_d)=1$. We will prove that if 
		$P_\Lambda$ is a contraction on $L^p(\mathbb{T}^d)$, then $\Lambda$ 
		must contain all integer multiples of $\gcd(a,b)\gamma$. In particular, 
		$\alpha+\beta = (a+b)\gamma$ will be in $\Lambda$. 
		
		If $n\gcd(a,b)\gamma$ and $(n+1)\gcd(a,b)\gamma$ are in $\Lambda$, then 
		we may appeal to Lemma~\ref{lem:Tdline} to see that $(n+2) 
		\gcd(a,b)\gamma$ and $(n-1) \gcd(a,b)\gamma$ must be in $\Lambda$. 
		Hence it is sufficient to establish that $\gcd(a,b) \gamma$ is in 
		$\Lambda$. 
		
		To prove this, we use a modified Euclidean algorithm. We identify the 
		integer $n$ with the point $n \gcd(a,b) \gamma$ and start with the 
		integers $a_1=a/\gcd(a,b)$ and $b_1=b/\gcd(a,b)$. We may assume without 
		loss of generality that $0<a_1<b_1$, since if $a_1=b_1$, there is 
		nothing to do. By Lemma~\ref{lem:Tdline}, we know that $c_1 = 2a_1 - 
		b_1$ is in $\Lambda$. We also see that $\gcd(a_1,b_1)=\gcd(a_1,c_1)$ 
		and $0 \leq |c_1|<b_1$. If $c_1=0$, then $a_1|b_1$ and $\gcd(a,b)=a$. 
		If $|c_1|>0$, then $\max(a_1,b_1)>\max(a_1,|c_1|)$ and we repeat the 
		procedure starting with $a_1$ and $|c_1|$. 
	\end{proof}
	
	\subsection{Proof of Theorem~\ref{thm:TdprojHp} for 
		\texorpdfstring{$p=\infty$}{p=infty}} \label{subsec:infty} Since 
	$H^p(\mathbb{T}^d)$ is a subspace of $L^p(\mathbb{T}^d)$, we know from 
	Theorem~\ref{thm:Tdproj} that if $\Gamma$ is the restriction of a coset in 
	$\mathbb{Z}^d$ to $\mathbb{N}_0^d$, then $\Gamma$ is a contractive 
	projection set for $H^p(\mathbb{T}^d)$. In this case, $\Gamma = 
	\Lambda(\Gamma) \cap \mathbb{N}_0^d$, where we recall that 
	$\Lambda(\Gamma)$ denotes the coset generated by $\Gamma$.
	
	Let us take a look at how Lemma~\ref{lem:Tdline} and 
	Lemma~\ref{lem:Tdplane} can be applied in the context of 
	$H^p(\mathbb{T}^d)$. Pick three affinely independent points 
	$\alpha,\beta,\gamma$ from $\Gamma$. Consider the function $f(z) = c 
	\varepsilon \overline{z}+1+\varepsilon z$ from Lemma~\ref{lem:Tdline}. By 
	replacing $f$ by
	\[g(z) := z^\beta f\left(z^{\alpha-\beta}\right) = c \varepsilon 
	z^{2\beta-\alpha} + z^\beta + \varepsilon z^\alpha,\]
	we see that Lemma~\ref{lem:Tdline} implies that if $\Gamma$ is a 
	contractive projection set for $H^p(\mathbb{T}^d)$ and the point 
	$2\beta-\alpha$ is in $\mathbb{N}_0^d$, then it must be included in 
	$\Gamma$. Geometrically, $2\beta-\alpha$ is the point obtained by 
	\emph{linear reflection} of $\alpha$ through $\beta$. By similar 
	considerations starting from Lemma~\ref{lem:Tdplane}, we also find that if 
	the point $\alpha + (\beta-\alpha) + (\gamma-\alpha)$ is in 
	$\mathbb{N}_0^3$, then it must be included in $\Gamma$ whenever $\Gamma$ is 
	a contractive projection set. Geometrically, this new point is obtained by 
	\emph{triangular reflection} of $\alpha$ through $\beta$ and $\gamma$.
	
	Figure~\ref{fig:rqc} contains all the points obtained by linear and 
	triangular reflections starting from the set $\Gamma = 
	\{(3,0,0),(0,3,0),(1,1,1)\}$. From the figure, we see that the necessary 
	conditions derived from Lemma~\ref{lem:Tdline} and Lemma~\ref{lem:Tdplane} 
	provide no insight into whether this $\Gamma$ is a contractive projection 
	set for $H^p(\mathbb{T}^3)$. 
	
	Moreover, when comparing Figure~\ref{fig:rqc} and Figure~\ref{fig:nocorner} 
	(which are based on the same initial set $\Gamma$), we see that the linear 
	and triangular reflections in Figure~\ref{fig:rqc} correspond precisely to 
	the points in $E_1(\Gamma)$. This is not a coincidence. It is easy to 
	verify that every $1$-extension is the same as a linear reflection or a 
	triangular reflection. In the latter case, we can see this by rewriting
	\[\alpha + (\beta-\alpha) + 
	(\gamma-\alpha)=\beta+(\beta-\alpha)-(\beta-\gamma).\]
	Is it therefore possible to prove Theorem~\ref{thm:TdprojHp} (a) using 
	Lemma~\ref{lem:Tdline} and Lemma~\ref{lem:Tdplane}.
	
	To see what additional estimates are required to handle case (b) and (c) of 
	Theorem~\ref{thm:TdprojHp}, recall that every $\lambda$ in 
	$\Lambda(\Gamma)$ can be represented as 
	\begin{equation}\label{eq:lrepn} 
		\lambda = \gamma_0 + \sum_{j=1}^n m_j (\gamma_j-\gamma_0), 
	\end{equation}
	where $m_j$ are integers and $\{\gamma_0,\gamma_1,\ldots,\gamma_n\}$ is an 
	affinely independent subset in $\Gamma$ for $n = \dim(\Lambda(\Gamma))$. If 
	we hope to prove Theorem~\ref{thm:TdprojHp} by the same approach as 
	Theorem~\ref{thm:Tdproj}, we would require estimates for every 
	representation \eqref{eq:lrepn}. 
	
	In the case $p=\infty$, we may actually establish the additional estimates 
	in one fell swoop. This is especially fortunate since the duality 
	techniques that we will employ in the next section to study the case $1 
	\leq p < \infty$, $p\neq 2$, do not apply when $p=\infty$. 
	
	\begin{figure}
		\centering
		\begin{tikzpicture}
			\begin{axis}[
				axis equal image,
				xlabel= $x$, 
				ylabel = $y$,
				axis lines = center,
				xmin = -4,
				xmax = 7,
				ymin = -4,
				ymax = 7,
				every axis x label/.style={
					at={(ticklabel* cs:1.025)},
					anchor=west,
				},
				every axis y label/.style={
					at={(ticklabel* cs:1.025)},
					anchor=south,
				},
				ticks=none,
				axis line style={->}]
				
				\addplot[thin, name path=t1, draw opacity=0, domain=0:3, 
				samples=5] {3-x};
				\addplot[thin, name path=t2, draw opacity=0, domain=0:3, 
				samples=5] {0};
				\addplot[color=black, fill=black, fill opacity=0.25]
				fill between[of=t1 and t2, soft clip={domain=0:3},];
				
				\node[circle, draw, fill=black, scale=0.5] at (3,0){};
				\node[circle, draw, fill=black, scale=0.5] at (0,3){};
				\node[circle, draw, fill=black, scale=0.5] at (1,1){};
				
				\node[circle, draw, color=blue, fill=blue, scale=0.5, 
				opacity=1] at (-1,2){};
				\node[circle, draw, color=blue, fill=blue, scale=0.5, 
				opacity=1] at (2,-1){};
				
				\node[circle, draw, color=blue, fill=blue, scale=0.5, 
				opacity=1] at (5,-1){};
				\node[circle, draw, color=blue, fill=blue, scale=0.5, 
				opacity=1] at (-1,5){};
				\node[circle, draw, color=blue, fill=blue, scale=0.5, 
				opacity=1] at (6,-3){};
				\node[circle, draw, color=blue, fill=blue, scale=0.5, 
				opacity=1] at (-3,6){};
				
				\node[circle, draw, color=red, fill=red, scale=0.5, opacity=1] 
				at (2,2){};
				\node[circle, draw, color=red, fill=red, scale=0.5, opacity=1] 
				at (4,-2){};
				\node[circle, draw, color=red, fill=red, scale=0.5, opacity=1] 
				at (-2,4){};
				
			\end{axis}
		\end{tikzpicture}
		\caption{The points $\lambda$ obtained by {\color{blue}linear} and 
			{\color{red} triangular reflection} starting from the set 
			$\Gamma=\{(3,0,0),(0,3,0),(1,1,1)\}$, represented in the projected 
			plane defined by $z=3-x-y$. The shaded triangle represents the 
			intersection of this plane and the narrow cone. Note that none of 
			the points obtained are in $\mathbb{N}_0^3$ and that the point 
			$(0,0,3)$ is not obtained.}
		\label{fig:rqc}
	\end{figure}
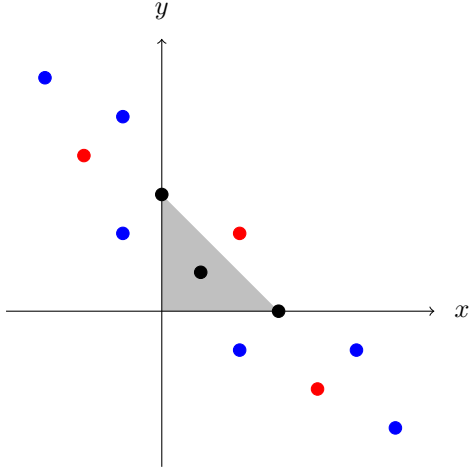
	
	\begin{lemma}\label{lem:infest} 
		Fix any $\alpha$ in $\mathbb{Z}^d$. Then 
		\begin{equation}\label{eq:infest} 
			\Bigg\|d + \sum_{j=1}^d z_j - \varepsilon z^\alpha \Bigg\|_\infty < 
			\Bigg\|d + \sum_{j=1}^d z_j\Bigg\|_\infty 
		\end{equation}
		for every sufficiently small $\varepsilon>0$. 
	\end{lemma}
	\begin{proof}
		The right-hand side of \eqref{eq:infest} is plainly equal to $2d$, so 
		it suffices to show that the left-hand side is strictly less than $2d$ 
		for some sufficiently small $\varepsilon>0$. By the triangle inequality, 
		we find that
		\[\Bigg|d + \sum_{j=1}^d z_j - \varepsilon z^\alpha \Bigg| \leq 
		2(d-1)+|1+z_j|+\varepsilon\]
		for any $j=1,2,\ldots,d$. Suppose that the supremum on the left-hand 
		side of \eqref{eq:infest} may be attained for $|1+z_j| \leq 
		2(1-\varepsilon)$ for some $j$. Then, clearly, the left-hand side is 
		equal to $2d-\varepsilon$, and we are done. Suppose therefore that
		\[\Bigg\|d + \sum_{j=1}^d z_j - \varepsilon z^\alpha \Bigg\|_\infty = 
		\sup_{\substack{z \in \mathbb{T}^d \\
				|1+z_j|\geq 2(1-\varepsilon)}}\Bigg|d + \sum_{j=1}^d z_j - 
		\varepsilon z^\alpha \Bigg|.\]
		To handle this case, we first estimate
		\[\sup_{\substack{z \in \mathbb{T}^d \\
				|1+z_j|\geq 2(1-\varepsilon)}}\Bigg|d + \sum_{j=1}^d z_j - 
		\varepsilon z^\alpha \Bigg| \leq 2(d-1) + \sup_{\substack{z \in 
				\mathbb{T}^d \\
				|1+z_j|\geq 2(1-\varepsilon)}} \left|1+z_1 - \varepsilon 
		z^\alpha\right| .\]
		Hence we are done if we can prove that 
		\begin{equation}\label{eq:onetwo} 
			\sup_{\substack{z \in \mathbb{T}^d \\
					|1+z_j|\geq 2(1-\varepsilon)}} \left|1+z_1 - \varepsilon 
			z^\alpha\right| < 2 
		\end{equation}
		for some sufficiently small $\varepsilon>0$. To this end, we see that 
		when $0<\varepsilon<1$, we have
		\[|1+z_j| \geq 2(1-\varepsilon) \qquad \Longleftrightarrow \qquad 
		|\theta_j| \leq 2\arccos(1-\varepsilon).\]
		Hence, if $|1+z_j| \geq (2-\varepsilon)$, then certainly $|\theta_j| 
		\leq 4 \sqrt{\varepsilon}$. If this estimate holds for every 
		$j=1,2,\ldots,d$ and $z^\alpha = e^{i\vartheta}$, then $|\vartheta| 
		\leq 4|\alpha| \sqrt{\varepsilon}$. By expanding and using Taylor's 
		theorem, we find that 
		\begin{align*}
			\left|1+z_j - \varepsilon z^\alpha\right|^2 &= 
			(1+\cos{\theta_j}-\varepsilon \cos{\vartheta})^2 + 
			(\sin{\vartheta_j}-\varepsilon \sin{\vartheta})^2 \\
			&= 2 + 2\cos(\theta_j) + \varepsilon^2 -2 \varepsilon 
			\big((1+\cos{\theta_j})\cos{\vartheta}-\varepsilon 
			\sin{\theta_j}\sin{\vartheta}\big) \\
			&= 4 - 4\varepsilon- \theta_j^2 + O(\varepsilon^2), 
		\end{align*}
		which establishes \eqref{eq:onetwo} for every sufficiently small 
		$\varepsilon>0$. 
	\end{proof}
	\begin{proof}
		[Proof of Theorem~\ref{thm:TdprojHp} for $p=\infty$.] Suppose that 
		$\Gamma$ is not the restriction of a coset in $\mathbb{Z}^d$ to 
		$\mathbb{N}_0^d$. Hence we can find $\lambda$ in $\left(\Lambda(\Gamma) 
		\cap \mathbb{N}_0^d\right) \setminus \Gamma$. By \eqref{eq:lrepn} we 
		write
		\[\lambda = \gamma_0 + \sum_{j=1}^n m_j (\gamma_j-\gamma_0),\]
		where $m_j$ are integers and $\{\gamma_0,\gamma_1,\ldots,\gamma_n\}$ is 
		an affinely independent set in $\Gamma$. Let
		\[f_1(z) = n + \sum_{j=1}^n z_j -\varepsilon z^{\alpha},\qquad f_2(z) 
        =n + \sum_{j=1}^n z_j\]
		be the functions from Lemma~\ref{lem:infest} with 
		$\alpha=(m_1,m_2,\ldots,m_n)$ and define
		\[g_i(z) = z^{\gamma_0} 
		f_i\left(z^{\gamma_1-\gamma_0},\,z^{\gamma_2-\gamma_0},\,\ldots,\,z^{
			\gamma_n-\gamma_0}\right), \quad i=1,2.\]
		Since $\{\gamma_0,\gamma_1,\ldots,\gamma_n\}$ is an affinely 
		independent set, the estimates of Lemma~\ref{lem:infest} imply that
		$\|g_1\|_\infty < \|g_2\|_\infty$. 
		Hence $\Gamma$ is not a contractive projection set for 
		$H^\infty(\mathbb{T}^d)$. 
	\end{proof}

	\section{Contractive projection sets for 
		\texorpdfstring{$H^p(\mathbb{T}^d)$}{Hp(Td)} with 
		\texorpdfstring{$1\leq p<\infty$}{1<=p<infty}} \label{sec:TdprojHp}
	
	\subsection{Overview} \label{subsec:org} This section is devoted to the 
	proof of Theorem~\ref{thm:TdprojHp} for $p < \infty$. We begin in the next 
	subsection by reformulating the problem in terms of duality. We then record 
	some immediate consequences, which include the proof of 
	Theorem~\ref{thm:shapirogeometric} and the verification of 
	Theorem~\ref{thm:TdprojHp} when $k=1$ and when $p$ not even integer.
	
	Section~\ref{subsec:duality} sets the stage for the most substantial part 
	of the proof of Theorem~\ref{thm:TdprojHp} which splits naturally into 
	three parts: 
	\begin{itemize}
		\item Section~\ref{subsec:cex}: The necessity of the conditions in part 
		(b) and (c); 
		\item Section~\ref{subsec:proofa}: The sufficiency of the case $d\geq 
		3$ and $k=2$; 
		\item Section~\ref{subsec:proofb}: The sufficiency of the cases $d=k=2$ 
		and $d=k=3$. 
	\end{itemize}
	The necessity part requires four examples, while the two sufficiency parts 
	rely on making appropriate extensions of a given subset of $\mathbb N_0^d$ 
	in terms of a sequence of $1$- or $2$-extensions. Both constructions are 
	quite intricate in the case of $2$-extensions, and they also differ 
	substantially. The arguments used in the case $d\geq 3$ and $k=2$ combine 
	geometric and arithmetic considerations, while those used in the case 
	$d=k=3$, relying on linear algebra, are more of a combinatorial nature. 
	Another notable distinction between the two cases is that the first deals 
	primarily with finite sets, while the second is concerned with extensions 
	of finite sets to infinite sets.
	
	\subsection{Duality reformulation with some immediate consequences} 
	\label{subsec:duality} The main tool for the case $p<\infty$ of 
	Theorem~\ref{thm:TdprojHp} is the following result.
	\begin{lemma}\label{lem:shapiro} 
		Fix $1 \leq p < \infty$ and $d\geq1$. A set of frequencies $\Gamma$ in 
		$\mathbb{N}_0^d$ is a contractive projection set for 
		$H^p(\mathbb{T}^d)$ if and only if 
		\begin{equation}\label{eq:shapiro} 
			\int_{\mathbb{T}^d} |f(z)|^{p-2} f(z) \, 
			\overline{z^\lambda}\,dm_d(z) = 0 
		\end{equation}
		for every $f(z) = \sum_{\gamma \in \Gamma} a_\gamma z^\gamma$ in 
		$H^p(\mathbb{T}^d)$ and every $\lambda$ in $\left(\Lambda(\Gamma) \cap 
		\mathbb{N}_0^d\right) \setminus \Gamma$. 
	\end{lemma}
	\begin{proof}
		A function $f$ in $L^p(\mathbb T^d)$ is said to be \emph{orthogonal} to 
		a closed subspace $Y$ of $L^p(\mathbb T^d)$ if
		\[\|f\|_p \leq \|f+h\|_p\]
		for every $h$ in $Y$. We will use the following characterization of 
		orthogonality due to Shapiro (see~\cite[Thm.~4.2.1 and 
		Thm.~4.2.2]{Shapiro}): a function $f$ is \emph{orthogonal} to $Y$ if 
		and only if
		\[\int_{\mathbb{T}^d} |f(z)|^{p-2} f(z) \, \overline{h(z)}\,dm_d(z) = 
		0,\]
		for every $h$ in $Y$. When $p = 1$, this holds if in addition the zero 
		set $\{f = 0\}$ has measure $0$, which will be the case because the 
		functions $f$ that we consider are in $H^1(\mathbb T^d)$, and thus 
		$\log|f|$ will be in $L^1(\mathbb T^d)$ (see 
		\cite[Thm.~3.3.5]{Rudin69}). We begin by proving the necessity of 
		\eqref{eq:shapiro}. Consider $f(z) = \sum_{\gamma \in \Gamma} a_\gamma 
		z^\gamma$ in $H^p(\mathbb{T}^d)$ and for any $\lambda$ in 
		$\left(\Lambda(\Gamma) \cap \mathbb{N}_0^d\right) \setminus \Gamma$ 
		take $Y$ to be the one-dimensional space spanned by $z^\lambda$. Since 
		$\Gamma$ is a contractive projection set, $\|f\|_p \leq \|f+ c 
		z^\lambda\|_p$ for any complex number $c$, thus $f$ is orthogonal to 
		$Y$, and \eqref{eq:shapiro} holds.
		
		To prove the reverse implication, we start by noting that since 
		$\Lambda(\Gamma)$ is a coset, $P_{\Lambda(\Gamma)}$ is a contraction on 
		$L^p(\mathbb T^d)$ by Theorem~\ref{thm:Tdproj}. Thus writing 
		$P_{\Gamma} = P_\Gamma P_{\Lambda(\Gamma)}$, we see that to prove that 
		$P_\Gamma$ is a contraction on $H^p(\mathbb{T}^d)$, we just need to show 
		that for any $h$ in $H^p(\mathbb T^d)$ with Fourier coefficients 
		supported on $\Lambda(\Gamma) \cap \mathbb{N}_0^d$, we have 
		$\|P_{\Gamma}h\|_p \leq \|h\|_p$. In fact, since the polynomials form a 
		dense subset of $H^p(\mathbb T^d)$ and $P_{\Lambda(\Gamma)} g$ is a 
		polynomial whenever $g$ is a polynomial, it suffices to prove this for 
		an arbitrary polynomial $h$. If we define $Y$ as the finite-dimensional 
		subspace of $H^p(\mathbb{T}^d)$ spanned by $\{z^\lambda\}$ for $\lambda$ 
		in the spectrum of $h$ minus $\Gamma$, we may decompose $h$ as 
		$h=P_\Gamma h+r$, where $r$ belongs to $Y$. By \eqref{eq:shapiro}, 
		$P_\Gamma h$ is orthogonal to $Y$, thus $\|P_{\Gamma}h\|_p \leq 
		\|P_{\Gamma}h+r \|_p$. 
	\end{proof}
	\begin{proof}
		[Proof of Theorem~\ref{thm:TdprojHp} for $p<\infty$ not an even 
		integer] If $\Gamma$ is not the restriction of a coset in $\mathbb{Z}^d$ 
		to $\mathbb{N}_0^d$, there is some $\lambda$ in $\left(\Lambda(\Gamma) 
		\cap \mathbb{N}_0^d\right) \setminus \Gamma$. Set 
		$n=\dim(\Lambda(\Gamma))\geq1$. There is an affinely independent subset 
		$\{\gamma_0,\gamma_1,\ldots,\gamma_n\}$ of $\Gamma$ which generates 
		$\Lambda(\Gamma)$. In particular, we may write
		\[\lambda = \gamma_0 + \sum_{j=1}^n m_j (\gamma_j-\gamma_0),\]
		where $m_j$ are integers. Since $\lambda$ is not in $\Gamma$, we may 
		assume without loss of generality that $m_1>0$ by reordering 
		$\{\gamma_0,\gamma_1,\ldots,\gamma_n\}$ if necessary. Similarly, we may 
		assume that there is some $1\leq k_0 \leq n$ such that $m_1,\ldots, 
		m_{k_0} \geq 0$ and $m_{k_0+1},\ldots, m_n<0$. We set
		\[m_+:= \sum_{j=1}^n \max(m_j,0), \qquad m_-:= -\sum_{j=1}^n 
		\min(m_j,0) \qquad \text{and} \qquad M:=m_++m_-.\]
		Our assumptions imply that $m_+\geq 1$. Set
		\[f(z) := z^{\gamma_0} + \varepsilon \sum_{j=1}^n z^{\gamma_j}\]
		for $0<\varepsilon<1/n$ and define $g(z) = \sum_{j=1}^n 
		z^{\gamma_j-\gamma_0}$. By the binomial series, we obtain
		\[|f(z)|^{p-2} = \sum_{k_1,k_2=0}^\infty 
		\binom{p/2-1}{k_1}\binom{p/2-1}{k_2} \varepsilon^{k_1+k_2} g(z)^{k_1} 
		g(\overline{z})^{k_2}.\]
		Since $p$ is not an even integer, none of the binomial coefficients 
		vanish. Writing
		\[f(z) = (1+\varepsilon g(z)) z^{\gamma_0},\]
		we see that
		\[F(\varepsilon) := \int_{\mathbb{T}^d} |f(z)|^{p-2} 
		f(z)\,\overline{z^\lambda}\,dm_d(z),\]
		is a non-trivial power series in $\varepsilon$. Indeed, we observe that
		\[F(\varepsilon)=\sum_{k=M}^\infty c_k \varepsilon^{k},\]
		where
		\[c_M= \binom{p/2-1}{m_-} 
		\binom{m_+}{m_1,\ldots,m_{k_0}}\binom{m_-}{|m_{k_0+1}|,\ldots,|m_n|} 
		\left( \binom{p/2-1}{m_+} +\binom{p/2-1}{m_+-1}\right)\]
		which evidently is nonzero. Consequently, there is some 
		$0<\varepsilon<1/n$, such that $F(\varepsilon)\neq0$. We invoke 
		Lemma~\ref{lem:shapiro} to conclude that $\Gamma$ is not a contractive 
		projection set. 
	\end{proof}
	
	It remains to deal with the most difficult case, which is when $p=2(n+1)$ 
	for some non-negative integer $n$. We begin by establishing 
	Theorem~\ref{thm:shapirogeometric}, which is a geometric reformulation of 
	Lemma~\ref{lem:shapiro}.
	\begin{proof}
		[Proof of Theorem~\ref{thm:shapirogeometric}] We will use 
		Lemma~\ref{lem:shapiro}. Let $\Gamma_0$ be any finite subset of 
		$\Gamma$ and consider the polynomial
		\[f(z) = \sum_{\alpha \in \Gamma_0} z^\alpha.\]
		We fix some $\gamma$ in $\Gamma_0$ and study the Fourier coefficients 
		of $z^\gamma |f(z)|^{2n}$. By the binomial theorem
		\[|f(z)|^{2n} = \sum_{j,k=0}^n 
		\binom{n}{j}\binom{n}{k}\bigg(\sum_{\alpha \in \Gamma_0 
			\setminus\{\gamma\}} z^{\alpha-\gamma}\bigg)^j\bigg(\sum_{\beta \in 
			\Gamma_0 \setminus\{\gamma\}} z^{-(\beta-\gamma)}\bigg)^k.\]
		The binomial coefficients are strictly positive, so by expanding 
		further we see that $|f|^{2n}$ has strictly  positive Fourier 
		coefficients for the frequencies which may be represented by
		\[\sum_{\alpha \in \Gamma_0 \setminus\{\gamma\}} j_\alpha 
		(\alpha-\gamma) - \sum_{\beta \in \Gamma_0 \setminus\{\gamma\}} k_\beta 
		(\beta-\gamma)\]
		where the coefficients $j_\alpha$ and $k_\beta$ are non-negative 
		integers whose individual sums do not exceed $n$. Equivalently, we 
		obtain the exponents
		\[\sum_{\alpha \in \Gamma_0 \setminus\{\gamma\}} m_\alpha 
		(\alpha-\gamma) \qquad \text{for} \qquad \max\left(\sum_{m_\alpha>0} 
		m_\alpha, -\sum_{m_\alpha<0} m_\alpha\right) \leq n.\]
		It is evident that no other choice of $f$ supported on $\Gamma_0$ can 
		give more frequencies. Returning to \eqref{eq:shapiro}, we see that the 
		only possible $\lambda$ such that the integral is non-zero are those in 
		$E_n(\Gamma)$. The claim now follows from Lemma~\ref{lem:shapiro}. 
	\end{proof}
	
	By Theorem~\ref{thm:shapirogeometric}, our task is now to clarify under 
	which conditions on a subset $\Gamma$ of $\mathbb N_0^d$ we will have 
	$E_n(\Gamma)=\Gamma$. To this end, the following terminology will be 
	useful. 
	\begin{definition}
		Let $T$ be a subset of $\mathbb N^d_0$. Define inductively 
		$E_{n}^{k+1}(T):=E_n(E_n^k(T))$ for all positive integers $k$ and set
		\[E_n^\infty(T):=\bigcup_{k=1}^\infty E^k_n(T).\]
		We will refer to the set $E_n^{\infty}(T)$ as the \emph{$n$-completion 
			of} $T$. 
	\end{definition}
	
	Clearly, $E_n^{\infty}(T)$ is the smallest set $\Gamma$ satisfying 
	$T\subseteq \Gamma$ and $E_n(\Gamma)=\Gamma$. We close this subsection by 
	recording two immediate consequences, both pertaining to the simplest case 
	of $1$-completions. The first of these settles the essentially trivial case 
	$k=1$ in part (a) of Theorem~\ref{thm:TdprojHp}. 
	\begin{lemma}\label{lem:k=1} 
		Let $T$ be a subset of $\mathbb N_0^d$ with $\dim (T)=1$. Then the 
		$1$-completion of $T$ is $ \Lambda (T) \cap \mathbb N_0^d$. 
	\end{lemma}
	\begin{proof}
		The assertion is trivial if $T$ consists of only two points, so suppose 
		that there are at least three points in $T$. Choose two distinct points 
		$\alpha$ and $\beta$ in $E_1^\infty(T)$ subject to condition that the 
		vector $\alpha-\beta$ have minimal length. By assumption, there is at 
		least one more point $\eta$ in $E_1^\infty(T)$. Now it is clear that 
		$\eta=\alpha+k(\beta-\alpha)$ for some integer $k$ since otherwise we 
		could find a point $\tau$ in $E_1^{\infty}(\{\alpha, \beta\})$ so that 
		the length of $\eta-\tau$ is positive and strictly smaller than that of 
		$\alpha-\beta$. 
	\end{proof}
	
	The next lemma will be useful for the analysis of our examples in 
	Section~\ref{subsec:cex}. It will also be instrumental in 
	Section~\ref{subsec:proofb}, for the computation of $E_1^\infty(T)$ and 
	$E_2^\infty(T)$ for subsets $T$ of codimension $0$ in respectively $\mathbb 
	N^2_0$ and $\mathbb N^3_0$.
	\begin{lemma}\label{lem:pluscompletion} 
		Let $T$ be a subset of $\mathbb N^d_0$. If there are points $\alpha$ 
		and $\beta$ in $E_n^\infty(T)$ such that $\beta-\alpha$ is in $\mathbb 
		N^d$, then $E_n^\infty(T) = E_1^\infty(T\cup \{\alpha, \beta\}) = 
		\Lambda(T)\cap \mathbb N^d_0$. 
	\end{lemma}
	\begin{proof}
		Let $\tau$ be any point in $\Lambda(T)\cap \mathbb N^d_0$. Then there 
		exist a positive integer $k$ and (not necessarily distinct) points 
		$\gamma_1,\ldots, \gamma_k$ in $T$, with a choice of signs 
		$\varepsilon_j$ such that
		\[\tau=\alpha+\sum_{j=1}^k \varepsilon_j (\gamma_j-\alpha).\]
		Now let $r$ be a positive integer which is so large that for $\eta:= 
		\alpha + r(\beta-\alpha)$, we have that $\eta+\sum_{j=1}^l 
		\varepsilon_j (\gamma_j-\alpha)$ lie in $\mathbb N_0^d$ for $l=1, 
		\ldots, k$. This implies that $\tau+r(\beta-\alpha)$ is in 
		$E_1^\infty(T\cup \{\alpha, \beta\})$, which in turn means that also 
		$\tau$ is in $E_1^\infty(T\cup \{\alpha, \beta\})$. 
	\end{proof}
	
	\subsection{Examples} \label{subsec:cex} Our goal is now to compile a 
	collection of examples which, in view of 
	Theorem~\ref{thm:shapirogeometric}, collectively demonstrate the necessity 
	part of points (b) and (c) of Theorem~\ref{thm:TdprojHp} in the case when 
	$p$ is an even integer. Two of the examples will also be used in the proof 
	of Theorem~\ref{thm:curlyop}. After each example, we will elucidate 
	explicitly its usage in the proof of Theorem~\ref{thm:TdprojHp}.
	
	We will make use of the following equivalent representation of the 
	$n$-extensions, which can be deduced from Lemma~\ref{lem:shapiro} similarly 
	to how we proved Theorem~\ref{thm:shapirogeometric}. Suppose that $\Gamma = 
	\{\gamma_1,\gamma_2,\ldots,\gamma_k\}$. A point $\lambda$ is in 
	$E_n(\Gamma)$ if and only if there are functions $\tau_+ \colon 
	\{1,\ldots,n+1\} \to \{\gamma_1,\ldots,\gamma_k\}$ and $\tau_- \colon 
	\{1,\ldots,n\} \to \{\gamma_1,\ldots,\gamma_k\}$ such that 
	\begin{equation}\label{eq:reform} 
		\lambda = \sum_{m=1}^{n+1} \gamma_{\tau_{+}(m)} - \sum_{m=1}^n 
		\gamma_{\tau_{-}(m)}. 
	\end{equation}
	The formulation \eqref{eq:reform} is particularly useful for checking if a 
	given $\lambda$ is in $E_n(\Gamma)$.
	
	The following example is presented graphically in Figure~\ref{fig:nocorner}.
	\begin{example}\label{ex:d3} 
		Consider $\Gamma := \{(3,0,0),(0,3,0),(1,1,1)\}$. We may represent 
		every $\lambda$ in $\Lambda(\Gamma)$ as 
		\begin{equation}\label{eq:d3} 
			\lambda = (1,1,1) + j (2,-1,-1) + k(-1,2,-1) 
		\end{equation}
		for integers $j$ and $k$. We are only interested in $\lambda$ that lie 
		in $\mathbb{N}_0^3$. We see that this can only be achieved if $j+k \leq 
		1$ by inspecting the third coordinate of \eqref{eq:d3} and $j,k\geq -1$ 
		by inspecting the first and second coordinates of \eqref{eq:d3}. 
		The only choice of $j$ and $k$ that provides a new point in 
		$\mathbb{N}_0^3$, is $j=k=-1$ which gives $\lambda = (0,0,3)$. Hence we 
		conclude that
		\[\Lambda(\Gamma) \cap \mathbb{N}_0^3 = \Gamma \cup \{(0,0,3)\}.\]
		Returning to \eqref{eq:d3} with $j=k=-1$, we see that $(0,0,3)$ is in 
		$E_2(\Gamma)$, which shows that 
		$E_2(\Gamma)=\Lambda(\Gamma)\cap\mathbb{N}_0^3$. It remains to show 
		that $E_1(\Gamma)=\Gamma$. In view of the reformulation 
		\eqref{eq:reform} and the discussion above, this is equivalent to 
		showing that the equation
		\[(0,0,3) = \gamma_1 + \gamma_2 - \gamma_3\]
		does not have a solution for (not necessarily distinct) 
		$\gamma_1,\gamma_2,\gamma_3$ in $\Gamma$. To see this, it is sufficient 
		to note that the third coordinate of $\gamma_1 + \gamma_2 - \gamma_3$ 
		is at most $2$. 
	\end{example}
	
	Example~\ref{ex:d3} extends trivially to an example for $d\geq 3$ if we 
	retain the first three entries as above and set the $j$th entry to $0$ for 
	$3\leq j \leq d$. This means that this example yields the necessity of the 
	case $k=2$ in part (b) of Theorem~\ref{thm:TdprojHp}.
	\begin{example}\label{ex:d3k3} 
		Consider $\Gamma := \{(4,0,0),(0,4,0),(0,0,4),(1,1,1)\}$. It is clear 
		that the only way to get $\gamma_1+\gamma_2-\gamma_3$ in 
		$\mathbb{N}_0^3$ for (not necessarily distinct) 
		$\gamma_1,\gamma_2,\gamma_3$ in $\Gamma$ is to set $\gamma_1=\gamma_3$ 
		or $\gamma_2=\gamma_3$. Hence $E_1(\Gamma)=\Gamma$. However, since
		\[(4,0,0)+\big((0,4,0)-(1,1,1)\big)+\big((0,0,4)-(1,1,1)\big)=(2,2,2)\]
		we conclude that $(2,2,2)$ is in $E_2(\Gamma)$. Since $(2,2,2)-(1,1,1)$ 
		is in $\mathbb{N}^3$, we get from Lemma~\ref{lem:pluscompletion} that 
		$E_2^\infty(\Gamma)=\Lambda(\Gamma)\cap\mathbb{N}_0^3$. 
	\end{example}
	
	We see that Example~\ref{ex:d3k3} settles the necessity of the case $d=k=3$ 
	in part (b) of Theorem~\ref{thm:TdprojHp}.
	\begin{example}\label{ex:d4} 
		Fix an integer $n\geq2$ and consider
		\[\Gamma_n := 
		\left\{(n,1,0,1),(n+1,0,1,0),(0,0,n+1,0),(0,0,0,n+1)\right\}.\]
		The generating vectors for the coset $\Lambda(\Gamma_n)$ with respect 
		to $\alpha :=(n,1,0,1)$ are 
		\begin{align*}
			v_1 &:= (1,-1,1,-1), \\
			v_2 &:= (-n,-1,n+1,-1), \\
			v_3 &:= (-n,-1,0,n). 
		\end{align*}
		Hence, every $\lambda$ in $\Lambda(\Gamma_n)$ may be represented as 
		\begin{equation}\label{eq:d4eq1} 
			\lambda = \alpha + j_1 v_1 + j_2 v_2 + j_3 v_3 
		\end{equation}
		for integers $j_1,j_2,j_3$. We want to check whether there are 
		$\lambda$ in $\mathbb{N}_0^4$ that satisfy the equation 
		\eqref{eq:d4eq1}. This means that we require 
		\begin{align}
			n + j_1-nj_2-nj_3 & \geq 0, \label{eq:d411} \\
			1 - j_1-j_2-j_3 & \geq 0, \label{eq:d412} \\
			j_1 + (n+1) j_2 & \geq 0, \label{eq:d413} \\
			1 - j_1 - j_2 +n j_3 & \geq 0. 
			\label{eq:d414} \end{align}
		We divide our analysis into four cases. 
		
		\subsubsection*{Case 1} Suppose that $j_1=1$. By \eqref{eq:d413}, we 
		get $j_2\geq0$. Rewriting \eqref{eq:d412} as $j_3 \leq - j_2$ and 
		inserting this into \eqref{eq:d414}, we find the necessary condition 
		$-(n+1)j_2 \geq0$. Hence $j_2 \leq 0$ and so $j_2=0$. Returning to 
		\eqref{eq:d412} and \eqref{eq:d414} we find that $j_3=0$. We get
		\[\lambda = \alpha + v_1 = (n,1,0,1).\]
		
		\subsubsection*{Case 2} Suppose that $j_1=0$. By \eqref{eq:d413}, we 
		get $j_2\geq0$. Rewriting \eqref{eq:d412} as $j_3 \leq 1-j_2$ and 
		inserting this into \eqref{eq:d414}, we find the necessary condition 
		$(1-j_2)(n+1)\geq0$. Hence $j_2 \leq 1$ and so either $j_2=0$ or 
		$j_2=1$. Returning to \eqref{eq:d412} and \eqref{eq:d414}, we see at 
		once that if $j_2=0$, then $j_3=0,1$ and if $j_2=1$, then $j_3=0$. The 
		three cases give 
		\begin{align*}
			\lambda &= \alpha = (n+1,0,1,0), \\
			\lambda &= \alpha+v_2 = (0,1,n,0), \\
			\lambda &= \alpha+v_3=(0,0,0,n+1). 
		\end{align*}
		
		\subsubsection*{Case 3} Suppose that $j_1>1$. Rewriting \eqref{eq:d412} 
		as $j_3 \leq 1 - j_1 - j_2$ and inserting this into \eqref{eq:d414}, we 
		obtain the necessary condition
		\[0\leq (1-j_1-j_2)(n+1),\]
		which means that $1-j_2 \geq j_1$. Since $j_1>1$ we conclude that 
		$j_2<0$. From \eqref{eq:d413} we see that $j_1 \geq - (n+1) j_2$. Hence 
		we need $j_2<0$ to satisfy
		\[1-j_2 \geq -(n+1)j_2,\]
		which is impossible since $n\geq2$.
		
		\subsubsection*{Case 4} Suppose that $j_1<0$. By \eqref{eq:d413}, we 
		find that $j_2 \geq 1$. Inserting this into \eqref{eq:d411} and 
		\eqref{eq:d414}, we find that 
		\begin{align*}
			0 &\leq j_1 - nj_3, \\
			0 &\leq n j_3 -j_1, 
		\end{align*}
		whence $j_1 = n j_3$. Returning to \eqref{eq:d411}, we find that 
		$n(1-j_2)\geq0$ which means that $j_2 \leq 1$ and hence $j_2=1$. 
		Returning to \eqref{eq:d413}, we see that $j_1 + n+1 \geq 0$ and since 
		$j_1$ is a strictly negative multiple of $n$, we must have $j_1=-n$ and 
		hence $j_3=-1$. This gives the solution
		\[\lambda = \alpha - n v_1 + v_2 - v_3 = (0,n+1,1,0).\]
		Note that here we have used an $(n+1)$-extension, since $j_1+j_2=n+1$ 
		and $j_3=-1$.
		
		\subsubsection*{Final part} We have demonstrated that
		\[\Lambda(\Gamma_n)\cap \mathbb{N}_0^3 = \Gamma_n \cup \{(0,n+1,1,0)\} 
		= E_{n+1}(\Gamma_n).\]
		It remains to establish that $E_n(\Gamma_n) = \Gamma_n$. We want to 
		prove that it is impossible to write $\lambda=(0,n+1,1,0)$ in the 
		representation \eqref{eq:reform}. We begin by looking at the more 
		general equation
		\[(0,n+1,1,0)= k_1(n,1,0,1)+ 
		k_2(n+1,0,1,0)+k_3(0,0,n+1,0)+k_4(0,0,0,n+1)\]
		for arbitrary integers $k_1,k_2,k_3,k_4$. The second coordinate shows 
		that $k_1=n+1$. In the first coordinate this gives that $k_2=-n$. In 
		the third coordinate we find that $k_3=1$ and in the fourth coordinate 
		we find that $k_4=-1$. Hence the only solution is $k_1=n+1$, $k_2=-n$, 
		$k_3=1$ and $k_4=-1$. However, this is not of the representation 
		\eqref{eq:reform} since $k_1+k_3=n+2>n+1$. Hence $(0,n+1,1,0)$ is not 
		in $E_n(\Gamma_n)$, which shows that $E_n(\Gamma_n)=\Gamma_n$. 
	\end{example}
	
	Example~\ref{ex:d4} extends trivially to an example for $d\geq 5$ if we 
	keep the four first entries as above and set the $j$th entry to $0$ for 
	$5\leq j \leq d$. Hence this example yields the necessity of the case $k=3$ 
	in part (c) of Theorem~\ref{thm:TdprojHp}.
	\begin{example}\label{ex:d4k4} 
		Fix an integer $n\geq3$ and consider
		\[\Gamma_n := \left\{(n,1,0,1), (n+1,0,1,0), (0,0,n+1,0), (0,0,0,n+1), 
		(0,n+1,0,0)\right\}.\]
		The first four points in $\Gamma_n$ are the same as in 
		Example~\ref{ex:d4}, so we know that $(0,n+1,1,0)$ is in 
		$E_{n+1}(\Gamma_n)$. Using this point and the fifth point in $\Gamma_n$ 
		to perform $n+2$ successive $1$-extensions we conclude that
		\[(n,1,0,1)+(n+2)\big((0,n+1,1,0)-(0,n+1,0,0)\big)=(n,1,n+2,1)\]
		is in $E_{n+1}^\infty(\Gamma_n)$. Since 
		$(n,1,n+2,1)-(0,0,n+1,0)=(n,1,1,1)$ is in $\mathbb{N}^4$ we can appeal 
		to Lemma~\ref{lem:pluscompletion} to conclude that 
		$E_{n+1}^\infty(\Gamma_n)=\Lambda(\Gamma_n)\cap\mathbb{N}_0^4$. 
		
		To investigate $E_n(\Gamma_n)\setminus \Gamma_n$ we look at points in 
		$\mathbb{N}_0^d$ which may be represented as
		\[k_1(n,1,0,1)+ k_2(n+1,0,1,0)+ k_3(0,0,n+1,0)+ k_4(0,0,0,n+1)+ 
		k_5(0,n+1,0,0)\]
		where the integers $k_1,k_2,k_3,k_4,k_5$ must be chosen in accordance 
		with \eqref{eq:reform}. In particular, $-n \leq k_1,k_2,k_3,k_4,k_5 
		\leq n+1$ and $k_1+k_2+k_3+k_4+k_5=1$. By the analysis in 
		Example~\ref{ex:d4} we may restrict our attention to the case that 
		$k_5\neq 0$. 
		
		If $k_5\geq1$, then $k_1,k_2,k_3,k_4 \leq n$ which implies that 
		$k_3,k_4\geq0$ and $k_1+k_2\leq0$. Looking at the first coordinate we 
		get the condition
		\[k_1 n + k_2(n+1) = (k_1 + k_2)n+k_2\geq 0.\]
		By the requirements above, this is only possible if $k_1=-k_2$ and 
		$k_2\geq0$. Since now $k_1+k_2=0$ we get that $k_3=k_4=0$. The fourth 
		coordinate is currently equal to $-k_1$, which means that $k_1=0$ and 
		hence $k_2=0$. We get $(0,n+1,0,0)$ which is already in $\Gamma_n$. 
		
		If $k_5 \leq -1$, the second coordinate shows that $k_5=-1$ and 
		$k_1=n+1$. We now get from \eqref{eq:reform} that $k_2,k_3,k_4 \leq 0$ 
		and $k_2+k_3+k_4=-(n-1)<0$. By looking at the third coordinate, we find 
		that $k_2=k_3=0$. Hence $k_4=-(n-1)$, so the fourth coordinate is
		\[n+1-(n-1)(n+1)=2+n-n^2<0\]
		since $n\geq3$. Hence $E_n(\Gamma_n)=\Gamma_n$. 
	\end{example}
	
	When $d\geq 4$, Example~\ref{ex:d4k4} allows us to settle the necessity of 
	the case $4\leq k \leq d$ in part (c) of Theorem~\ref{thm:TdprojHp}. This 
	is immediate if $d=4$, and for $d\geq 5$ we make the following trivial 
	extension. We retain the first four entries as above the points and put the 
	$j$th entry to $0$ for $5\leq j \leq d$, and then we extend the set by 
	adding $d-k$ affinely independent points with only zeros in the first $4$ 
	entries.
	
	\subsection{Two-dimensional subsets of 
		\texorpdfstring{$\mathbb{N}^d_0$}{Nd0} for \texorpdfstring{$d\geq 
			3$}{d=>3}} \label{subsec:proofa}
	
	The purpose of this subsection is to settle the sufficiency of the case 
	$k=2$ in part (b) of Theorem~\ref{thm:TdprojHp}. In view of 
	Theorem~\ref{thm:shapirogeometric}, this will be furnished by 
	Lemma~\ref{lem:lem78} below. We begin with the following special case of 
	the required result.
	\begin{lemma}\label{lem:lem77} 
		Let $T$ be a set of three affinely independent points in 
		$\mathbb{N}^d_0$ for $d\geq 3$. Then the $2$-completion of $T$ is 
		$\Lambda(T)\cap \mathbb{N}_0^d$. 
	\end{lemma}
	\begin{proof}
		Let $\alpha_1$, $\alpha_2$, $\alpha_3$ be the points in $T$, and let 
		$\beta$ be a point in $\Lambda(T)\cap \mathbb{N}_0^d$. We denote the 
		plane of which $T$ is a subset by $P(T)$. We let $\ell(\gamma,\tau)$ be 
		the line through the two points $\gamma$ and $\tau$, and we let 
		$\Delta(\gamma,\tau,\eta)$ be the triangle with corners $\gamma$, 
		$\tau$, $\eta$. Let $V$ and $W$ be the two components of $P(T)\setminus 
		\ell(\alpha_2,\beta)$. We may assume that the remaining two points 
		$\alpha_1$ and $\alpha_3$ lie in either $\overline{V}$ or 
		$\overline{W}$. We may also assume that $\alpha_3$ is contained in the 
		closed strip lying between $\ell(\alpha_2,\beta)$ and the line through 
		$\alpha_1$ parallel to $\ell(\alpha_2,\beta)$, since otherwise it could 
		be moved into this strip by a finite number of $1$-extensions. In fact, 
		we may assume that $\alpha_3$ lies in the interior of this strip, 
		because the problem has a trivial solution in terms of a finite number 
		of $1$-extensions should $\alpha_3$ lie on the boundary of the strip.
		
		Now if $\alpha_3$ lies in the parallelogram with corners $\alpha_1$, 
		$\alpha_2$, $\beta$, $\beta+\alpha_1-\alpha_2$, then it must lie inside 
		the triangle $\Delta(\beta,\alpha_1,\alpha_2)$, since otherwise $\beta$ 
		would not be contained in $\Lambda(T)$. If $\alpha_3$ does not lie in 
		this parallelogram, then we may replace $\alpha_3$ by 
		$\alpha_1+\alpha_2-\alpha_3$ which then must lie inside 
		$\Delta(\beta,\alpha_1,\alpha_2)$. We may therefore assume that 
		$\alpha_3$ lies inside $\Delta(\beta,\alpha_1,\alpha_2)$. 
		
		Let $\eta$ be the point at which $\ell(\alpha_1,\alpha_3)$ and 
		$\ell(\alpha_2,\beta)$ intersect. Since $\beta$ is in $\Lambda(T)$, the 
		distance from $\alpha_2$ to $\eta$ divides the distance from $\alpha_2$ 
		to $\beta$, whence $\beta-\alpha_2=n(\eta-\alpha_2)$ for a positive 
		integer $n$. This means that 
		\begin{equation}\label{eq:ab} 
			\alpha_3-\alpha_1 =a(\alpha_2-\alpha_1)+b(\beta-\alpha_2), 
		\end{equation}
		where $a$ and $b$ are two positive rational numbers such that 
		$b=1/m\leq 1/n$ and $a=n/m$. We see that then 
		\begin{equation}\label{eq:y1} 
			\beta=\alpha_2+ m(\alpha_3- \alpha_1)- n(\alpha_2-\alpha_1). 
		\end{equation}
		We may assume that $(m,n)=1$ since otherwise we could replace $m$ and 
		$n$ by respectively $m/(m,n)$ and $n/(m,n)$. Figure~\ref{fig:figY} 
		illustrates the case when $m=5$ and $n=2$. We need to show that we can 
		get to $\beta$ starting from $T=\{\alpha_1,\alpha_2,\alpha_3\}$ and 
		using $2$-extensions. Reformulating \eqref{eq:y1} to
		\[\beta= \alpha_3 + (m-1)(\alpha_3-\alpha_1) - (n-1)(\alpha_2 - 
		\alpha_1),\]
		we may exclude from our discussion the case when $m,n\leq 3$, since we 
		may evidently reach $\beta$ directly from $\alpha_3$ using a single 
		$2$-extension.
		
		\begin{figure}
			\centering
			\begin{tikzpicture}
				\begin{axis}[
					axis equal image,
					axis lines = none,
					xmin = -3,
					xmax = 13,
					ymin = -3,
					ymax = 22]
					
					\addplot[thin, name path=t1, draw opacity=0] coordinates 
					{(0,0) (10,20)};
					\addplot[thin, name path=b1, draw opacity=0] coordinates 
					{(0,0) (10,0)};
					\addplot[color=black, fill=black, fill opacity=0.25] fill 
					between[of=t1 and b1];
					
					\addplot[solid] coordinates {(0,0) (10,0)};
					\addplot[solid] coordinates {(0,0) (8,8)};
					\addplot[solid] coordinates {(20,0) (16,19)};
					
					\addplot[densely dashed, blue] coordinates {(10,20) (6,16)};
					\addplot[densely dashed, blue] coordinates {(8,8) (6,16)};
					\addplot[densely dashed] coordinates {(-2,8) (8,8)};
					\addplot[densely dashed] coordinates {(-2,8) (6,16)};
					
					\node[circle, draw, fill=black, scale=0.5] at (0,0){};
					\node[] at (0,-2){$\alpha_1$};
					\node[circle, draw, fill=black, scale=0.5] at (10,0){};
					\node[] at (10,-2){$\alpha_2$};
					\node[circle, draw, fill=black, scale=0.5] at (4,4){};
					\node[] at (5.4142,2.5857){$\alpha_3$};
					
					\node[circle, draw, color=red, fill=red, scale=0.5] at 
					(10,20){};
					\node[] at (12,20){\color{red}$\beta$};
					\node[circle, draw, color=blue, fill=blue, scale=0.5] at 
					(8,8){};
					\node[] at (9.4142,6.5857){\color{blue}$\xi$};
					\node[circle, draw, scale=0.5] at (10,10){};
					\node[] at (12,10){$\eta$};
					
				\end{axis} 
			\end{tikzpicture}
			\caption{The case $m=5$ and $n=2$ in the proof of 
				Lemma~\ref{lem:lem77}. The shaded area represents a part of the 
				plane $P(T)$ that must lie inside the narrow cone $\mathbb 
				N_0^d$. We need a $2$-extension to reach ${\color{red}\beta}$ 
				which is accommodated by the move ${\color{red}\beta} =
				{\color{blue}\xi}-(\alpha_2-{\color{blue}\xi})-(\alpha_3-{\color 
					{blue}\xi} )$.}
			\label{fig:figY}
		\end{figure}
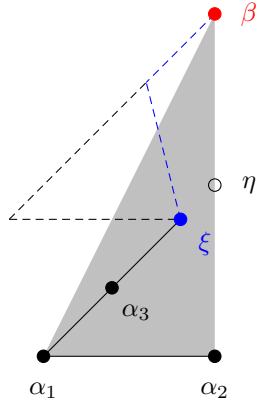
		
		Our plan is now to make successive extensions so that the point that is 
		added in each step, lies inside $\Delta(\beta,\alpha_1,\alpha_2)$. 
		Using \eqref{eq:ab}, we see that the condition that a point of the form 
		\begin{equation}\label{eq:ksi} 
			\xi=\alpha_2 + j(\alpha_3 - \alpha_1) - k(\alpha_2-\alpha_1) 
		\end{equation}
		for $j,k>0$ to be inside $\Delta(\beta,\alpha_1,\alpha_2)$ is that 
		$j/m\leq jn/m+1-k\leq 1$. This means more specifically that 
		\begin{equation}\label{eq:cond} 
			k=\left\lceil \frac{jn}{m}\right\rceil \qquad \text{and} \qquad 
			\left\{\frac{jn}{m}\right\}\geq \frac{j}{m}. 
		\end{equation}
		We will begin by identifying what will be the final extension required 
		to reach $\beta$. The basic idea is that it suffices with one final 
		extension once we have reached essentially half way from the base of 
		the triangle $\Delta(\alpha_1,\alpha_2,\beta)$ to the corner at 
		$\beta$. We make this precise by distinguishing between the following 
		three cases: 
		\subsubsection*{Case 1} If $m$ is an even number, then it is clear by 
		\eqref{eq:cond} that
		\[\xi=\alpha_2+ \frac{m}{2}(\alpha_3-\alpha_1) -\frac{(n+1)}{2}( 
		\alpha_2-\alpha_1)\]
		is in $\Delta(\beta,\alpha_1,\alpha_2)$. Assuming that $\xi$ is in 
		$E_2^k(T)$ for some $k$, we see that $\beta$ is in $E_2^{k+1}(T)$ by 
		recalling \eqref{eq:y1} and observing that
		\[\beta=\alpha_2+2(\xi-\alpha_2)-(\alpha_1-\alpha_2).\]
		\subsubsection*{Case 2} If $m$ and $n$ are both odd numbers, then
		\[\left\{\frac{(m+1)n}{2m}\right\} = 
		\left\{\frac{1}{2}+\frac{n}{2m}\right\}=\frac{(m+n)}{2m} \geq 
		\frac{(m+1)}{2m},\]
		whence
		\[\xi=\alpha_2+ \frac{(m+1)}{2}(\alpha_3-\alpha_1)-\frac{(n+1)}{2}( 
		\alpha_2-\alpha_1)\]
		is in $\Delta(\beta,\alpha_1,\alpha_2)$ by \eqref{eq:cond}. Assuming 
		again that $\xi$ is in $E_2^k(T)$ for some $k$, we now find that 
		$\beta$ is in $E_2^{k+1}(T)$ by recalling \eqref{eq:y1} and observing 
		that
		\[\beta=\alpha_2+2(\xi-\alpha_2)-(\alpha_3-\alpha_2).\]
		
		\subsubsection*{Case 3} The case when $m$ is an odd number and $n$ is 
		an even number requires a slightly more refined analysis. To begin with, 
		we observe that
		\[\left\{\frac{(m-1)n}{2m}\right\} = 
		\left\{-\frac{n}{2m}\right\}=\frac{(2m-n)}{2m} \geq \frac{(m-1)}{2m},\]
		whence
		\[\xi_1:=\alpha_2+\frac{(m-1)}{2}(\alpha_3-\alpha_1) 
		-\frac{n}{2}(\alpha_2-\alpha_1)\]
		is in $\Delta(\beta,\alpha_1,\alpha_2)$ by \eqref{eq:cond}. We use next 
		that
		\[\left\{\frac{(m \pm 3)n}{2m}\right\} = \left\{\frac{\pm 
			3n}{2m}\right\}.\]
		We observe that if $m/n>3/2$, then at least one of the two inequalities
		\[\left\{\frac{3n}{2m}\right\} \geq \frac{m+3}{2m} 
		\qquad\text{and}\qquad \left\{\frac{-3n}{2m}\right\} \geq 
		\frac{m-3}{2m}\]
		must hold. On the other hand, if $m/n<3/2$, then
		\[\left\{\frac{-3n}{2m} \right\} = \frac{4m-3n}{2m} \geq 
		\frac{m-3}{2m}.\]
		We conclude that at least one of the three points 
		\begin{align*}
			\xi_2 & := \alpha_2+\frac{(m+1)}{2} (\alpha_3-\alpha_1) - 
			\left(\frac{n}{2}+1\right)(\alpha_2-\alpha_1), \\
			\xi_3 & := \alpha_2+\frac{(m-3)}{2} (\alpha_3-\alpha_1) - 
			\frac{n}{2}(\alpha_2-\alpha_1), \\
			\xi_4 & := \alpha_2+\frac{(m-3)}{2} (\alpha_3-\alpha_1) - 
			\left(\frac{n}{2}-1\right)(\alpha_2-\alpha_1), 
		\end{align*}
		is in $\Delta(\beta,\alpha_1,\alpha_2)$ by \eqref{eq:cond}. Assuming 
		first that $\xi_1$ and $\xi_2$ are in $E_2^k(T)$ for some $k$, we see 
		that $\beta$ is in $E_2^{k+1}(T)$ because
		\[\beta=\alpha_2+(\xi_1-\alpha_2)+(\xi_2-\alpha_2)-( 
		\alpha_1-\alpha_2).\]
		Next, if $\xi_1$ and $\xi_3$ are in $E_2^k(T)$ for some $k$, then we 
		find again that $\beta$ is in $E_2^{k+1}(T)$ because
		\[\beta=\xi_1-(\alpha_2-\xi_1)-(\xi_3-\xi_1).\]
		Finally, if $\xi_1$ and $\xi_4$ are in $E_2^k(T)$ for some $k$, then we 
		find as before that $\beta$ is in $E_2^{k+1}(T)$, this time because
		\[\beta=\xi_1-(\alpha_1-\xi_1)-(\xi_4-\xi_1).\]
		
		\subsubsection*{Final part} We are now left with the simpler problem of 
		reaching each of the points considered above. We claim that any one of 
		them can be reached by starting from $\alpha_1$ or $\alpha_3$ and 
		making successive additions of multiples of the vectors 
		$\alpha_3-\alpha_1$ and $\alpha_1-\alpha_2$. We will refer to the 
		integers $j$ and $k$ in \eqref{eq:ksi} as respectively steps and 
		levels. Notice that a step $j$ determines a unique point in the strip 
		between $\alpha_1-\alpha_2+\ell(\alpha_2,\beta)$ and 
		$\ell(\alpha_2,\beta)$, while there may in general be several points in 
		this strip at each level. Note that $\alpha_3$ is always step $1$. In 
		Figure~\ref{fig:figY}, the point $\xi$ corresponds for example to step 
		$2$ and the point $\beta$ is at level $2$.
		
		A simple geometric consideration suffices to settle the case $n<m/2$. 
		Indeed, then for every level $k\leq m/2$, there are points of the form 
		\eqref{eq:ksi} lying in $\Delta(\beta,\alpha_1,\alpha_2)$, and it is 
		clear that the points accumulated at the initial level $k=1$ can be 
		used to connect those lying at any level $k\leq m/2$ with those found at 
		the next level $k+1$. 
		
		The case $m/2<n<m$ requires a more careful analysis. It may be helpful 
		to bear in mind that the lead role will now be played by the steps $j$ 
		rather than the levels $k$. We begin by treating separately a special 
		case. Suppose that $m$ is odd and $n=(m+1)/2$. If $j$ is odd and $j<m$, 
		then
		\[\left\{\frac{jn}{m}\right\}=\frac{1}{2}+\frac{j}{2m} >\frac{j}{m},\]
		which means that each of the points $\xi$ in \eqref{eq:ksi} with $j$ 
		odd and $j<m$ will be in $\Delta(\beta,\alpha_1,\alpha_2)$. These points 
		are reached in an obvious way, once we have made the initial extension
		\[\xi:=\alpha_3-(\alpha_1-\alpha_3)-(\alpha_2-\alpha_3).\]
		
		We now write $n=m-r$ and assume that $1\leq r \leq m/2-1$. The 
		condition that the point $\xi$ in \eqref{eq:ksi} be in 
		$\Delta(\beta,\alpha_1,\alpha_2)$ is that
		\[\left\{\frac{jn}{m} \right\}=\left\{\frac{-jr}{m} \right\}\geq 
		\frac{j}{m}.\]
		This means that we must have 
		\begin{equation}\label{eq:interval} 
			\frac{(t-1)m}{r}<j\leq \frac{t m}{r+1} 
		\end{equation}
		for some $t$ such that $1 \leq t \leq r/2+1$, where the latter 
		inequality should hold because we require that $j\leq m/2+1$. We now 
		observe that, under this restriction, the interval defined by 
		\eqref{eq:interval} has length
		\[m \left(\frac{t}{r+1} -\frac{(t-1)}{r}\right)=\frac{m(r+1-t)}{ 
			r(r+1)}\geq \frac{m}{2(r+1)}\geq 1,\]
		whence it contains at least one integer. This yields an algorithm for 
		reaching all steps $j$ with $j\leq m/2+1$ such that $\xi$ in 
		\eqref{eq:ksi} is in $\Delta(\beta, \alpha_1,\alpha_2)$. Indeed, 
		initially we go step-by-step until $j=[m/(r+1)]$. (Notice that this 
		suffices when $r=1$.) We then observe, denoting the interval defined by 
		\eqref{eq:interval} by $I_{t}$, that
		\[\operatorname{dist}(I_{t+1},I_{t})=\frac{t m}{r(r+1)} \leq 
		\frac{m}{2(r+1)}\]
		when $t \leq r/2 $. This means that the points corresponding to the 
		steps $j\leq [m/(r+1)]$ can be used to connect those associated with 
		steps lying in $I_{t}$ to those lying in $I_{t+1}$. 
	\end{proof}
	
	The general case can now be settled with a proof that requires less effort 
	than the preceding one.
	
	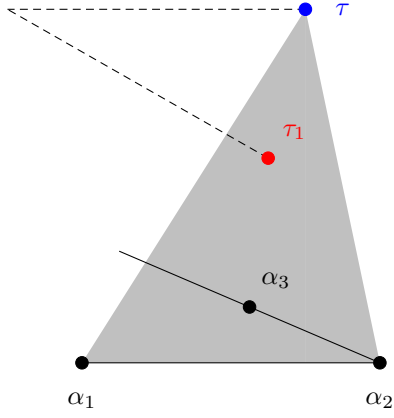
\begin{figure}
		\centering
		\begin{tikzpicture}
			\begin{axis}[
				axis equal image,
				axis lines = none,
				xmin = 0,
				xmax = 21,
				ymin = -3,
				ymax = 20]
				
				\addplot[thin, name path=t1, draw opacity=0] coordinates {(4,0) 
					(16,19)};
				\addplot[thin, name path=b1, draw opacity=0] coordinates {(4,0) 
					(16,0)};
				\addplot[color=black, fill=black, fill opacity=0.25] fill 
				between[of=t1 and b1];
				\addplot[thin, name path=t2, draw opacity=0] coordinates 
				{(16,19) (20,0)};
				\addplot[thin, name path=b2, draw opacity=0] coordinates 
				{(16,0) (20,0)};
				\addplot[color=black, fill=black, fill opacity=0.25] fill 
				between[of=t2 and b2];
				
				\addplot[solid] coordinates {(4,0) (20,0)};
				\addplot[solid] coordinates {(6,6) (20,0)};
				
				\addplot[densely dashed] coordinates {(14,11) (0,19)};
				\addplot[densely dashed] coordinates {(16,19) (0,19)};
				
				\node[circle, draw, fill=black, scale=0.5] at (4,0){};
				\node[] at (4,-2){$\alpha_1$};
				\node[circle, draw, fill=black, scale=0.5] at (20,0){};
				\node[] at (20,-2){$\alpha_2$};
				\node[circle, draw, fill=black, scale=0.5] at (13,3){};
				\node[] at (14.4142,4.4142){$\alpha_3$};
				
				\node[circle, draw, color=red, fill=red, scale=0.5] at 
				(14,11){};
				\node[] at (15.4142,12.4142){\color{red}$\tau_1$};
				\node[circle, draw, color=blue, fill=blue, scale=0.5] at 
				(16,19){};
				\node[] at (18,19){\color{blue}$\tau$};
				
			\end{axis} 
		\end{tikzpicture}
		\caption{The case $k=2$ and 
			${\color{red}\tau_1}={\color{blue}\tau}
			+(\alpha_1-\alpha_2)+k(\alpha_2-\alpha_3)$ in the proof of 
			Lemma~\ref{lem:lem78}. The shaded area represents a part of the 
			plane $P(T)$ that must lie inside the narrow cone $\mathbb{N}_0^d$.}
		\label{fig:figX}
	\end{figure}
	
	\begin{lemma}\label{lem:lem78} 
		Fix $d\geq3$ and let $T$ be a set in $\mathbb{N}^d_0$ with $\dim(T)=2$. 
		Then the $2$-completion of $T$ is $\Lambda(T)\cap \mathbb{N}_0^d$. 
	\end{lemma}
	\begin{proof}
		Lemma~\ref{lem:lem77} proves the assertion in the special case when the 
		cardinality of $T$ is $3$. We will use this result to run what may be 
		thought of as a kind of Euclidean algorithm. We pick three arbitrary 
		affinely independent points $\alpha_1$, $\alpha_2$, $\alpha_3$ in $T$ 
		and assume that $\tau$ is a fourth point in $T$ such that $\tau$ is not 
		in $\Lambda(\{\alpha_1,\alpha_2,\alpha_3\})$. The crucial point will be 
		to prove that, on this assumption, there exists a point $\beta$ in 
		$E_2^{\infty}(\{\tau,\alpha_1,\alpha_2,\alpha_3\})$ such that at least 
		one of the triangles $\Delta(\alpha_1,\alpha_2,\beta)$, 
		$\Delta(\alpha_1,\alpha_3,\beta)$, $\Delta(\alpha_2,\alpha_3,\beta)$, 
		say $\Delta(\alpha_1,\alpha_2,\beta)$ for definiteness, is 
		nondegenerate with area strictly smaller than that of 
		$\Delta(\alpha_1,\alpha_2,\alpha_3)$. This argument may be iterated so 
		that in the next step we use $\alpha_1, \alpha_2, \beta$ in place of 
		$\alpha_1, \alpha_2, \alpha_3$. The iteration must terminate after a 
		finite number of steps, which means that eventually there are no points 
		in $T$ lying outside of the coset generated by the points $\alpha_1$, 
		$\alpha_2$, $\alpha_3$ used in this final step of the iteration.
		
		We are now in a situation similar to that considered in the preceding 
		lemma. We have a trivial solution if $\tau$ lies on 
		$\ell(\alpha_1,\alpha_2)$: Then the desired $\beta$ lies on the segment 
		$[\alpha_1,\alpha_2]$ and is reached by a finite number of 
		$1$-extensions. We therefore ignore this case in what follows. We may 
		assume that $\tau$ and $\alpha_3$ lie in the same component of the set 
		$P(T)\setminus \ell(\alpha_1,\alpha_2)$. 
		
		We now let $m$ be the smallest positive integer such that $\tau$ is 
		contained in the open strip between $\ell(\alpha_1,\alpha_2)$ and 
		$m(\alpha_3-\alpha_2)+\ell(\alpha_1,\alpha_2)$. If neither $\alpha_3$ 
		nor any of the points $\alpha_3\pm (\alpha_2-\alpha_1)$ lie in the 
		triangle $\Delta(\alpha_1,\alpha_2,\tau)$, then our problem has a 
		trivial solution: For $k=m-1$ and $i=1$ or $i=2$, the point 
		$\beta:=\tau-k(\alpha_3-\alpha_i)$ will lie in the closure of 
		$\Delta(\alpha_1,\alpha_2,\tau)$ and have distance to 
		$\ell(\alpha_1,\alpha_2)$ strictly smaller than that from $\alpha_3$ to 
		$\ell(\alpha_1,\alpha_2)$. (This distance may be $0$.) This point 
		$\beta$ will therefore have the desired property, unless it lies on 
		$\ell(\alpha_1,\alpha_2)$ in which case the solution is again trivial 
		as we saw above. 
		
		What remains to consider is the case when $\alpha_3$ lies inside 
		$\Delta(\alpha_1,\alpha_2,\tau)$. Let $k$ be the smallest positive 
		integer such that $\alpha_2+k(\alpha_3-\alpha_2)$ does not lie in the 
		closure of $\Delta(\alpha_1,\alpha_2,\tau)$. If $k=m$, then we see that 
		$\beta=\tau+(m-1)(\alpha_2-\alpha_3)$ solves our problem. If $k<m$, 
		then the point
		\[\tau_1:= \tau +(\alpha_1-\alpha_2) +k(\alpha_2-\alpha_3)\]
		is in $\Delta(\alpha_1,\alpha_2,\tau)\cap 
		E_2^\infty(\{\alpha_1,\alpha_2,\alpha_3,\tau\})$. See 
		Figure~\ref{fig:figX}. Now $\tau_1$ lies in the open strip between 
		$\ell(\alpha_1,\alpha_2)$ and 
		$(m-k)(\alpha_3-\alpha_2)+\ell(\alpha_1,\alpha_2)$ or on 
		$\ell(\alpha_1,\alpha_2)$. We may thus iterate the argument with 
		$\tau_1$ in place of $\tau$. It is clear that after a finite number of 
		such iterations, we will reach a point $\tau_j$ in 
		$E_2^\infty(\{\alpha_1,\alpha_2,\alpha_3,\tau\})$ such that the desired 
		$\beta$ can be reached in any of the trivial ways described in the 
		preceding discussion. 
	\end{proof}
	
	\subsection{Subsets of \texorpdfstring{$\mathbb{N}^2_0$}{N02} and 
		\texorpdfstring{$\mathbb{N}^3_0$}{N03} of codimension 
		\texorpdfstring{$0$}{0}} \label{subsec:proofb} It remains to establish 
	the case $d=k=2$ in part (a) and to finish the case $d=k=3$ in part (b) of 
	Theorem~\ref{thm:TdprojHp}. In either case, we will be dealing with sets of 
	codimension $0$ in the ambient space.
	
	We begin with the easiest case $d=k=2$. By 
	Theorem~\ref{thm:shapirogeometric}, we need to show that 
	$E_1^\infty(T)=\Lambda(T)\cap \mathbb N_0^2$ when $T$ is a subset of 
	$\mathbb N^2_0$ with $\dim (T)=2$. In view of 
	Lemma~\ref{lem:pluscompletion}, this is accomplished by means of the 
	following lemma.
	\begin{lemma}\label{lem:positive2} 
		Let $T$ be a set of three affinely independent points in 
		$\mathbb{N}_0^2$. Then for every $\alpha$ in $T$ there exists a point 
		$\beta$ in $E^{\infty}_1(T)\setminus\{x\}$ such that $\beta-\alpha$ is 
		in $\mathbb{N}^2$. 
	\end{lemma}
	
	We will in the proof of this lemma and later, in its more elaborate 
	$3$-dimensional counterpart, make use of the following quantity.
	\begin{definition}
		Given a set $U$ of $d$ linearly independent vectors $u = 
		(u_1,\ldots,u_d)$ in $\mathbb{Z}^d$, we define the \emph{negativity 
			index of} $U$ as
		\[\ind(U):=\sum_{j=1}^d \min\big(0, \min_{u\in U}u_j\big).\]
	\end{definition}
	
	The vectors $u$ will be assumed to relate to a fixed point $\alpha$ in 
	$\mathbb{N}^d_0$ by the condition that $\alpha+u$ be in $\mathbb{N}_0^d$ as 
	well. When this holds, we say that $u$ is $\alpha$-\emph{admissible}. Our 
	goal will be to successively change $U$ by making $1$- or $2$-extensions of 
	$\alpha+U$ to get to new vectors with a larger negativity index. It will be 
	crucial that linear independence of the vectors of $U$ be preserved during 
	the course of this iteration.
	\begin{proof}
		[Proof of Lemma~\ref{lem:positive2}] We begin by noting that it 
		suffices to find a point $\beta$ in $E_1^{\infty}(T) \setminus T$ with 
		$\beta-\alpha$ in $\mathbb N_0^2$. Indeed, should one of the entries of 
		$\beta-\alpha$ be $0$, we may replace $\beta$ by either 
		$\alpha+m(\beta-\alpha)+(\tau-\alpha)$ or 
		$\alpha+m(\beta-\alpha)-(\tau-\alpha)$ for a sufficiently large $m$, 
		where $\tau$ is one of the two other points in $T$. Since $\dim (T)=2$, 
		both entries of either $m(\beta-\alpha)+(\tau-\alpha)$ or 
		$\alpha+m(y-\alpha)-(\tau-\alpha)$ will be positive for at least one 
		such $\tau$. It is plain that the corresponding point 
		$\alpha+m(\beta-\alpha)\pm (\tau-\alpha) $ will lie in $E_1^\infty(T)$.
		
		Now fix a point $\alpha$ in $T$, and let $v_1$, $v_2$ be the vectors 
		going from $\alpha$ to the other two points in $T$. It will be helpful 
		to represent an entry in any of the two vectors $v_1$, $v_2$ 
		symbolically by $+$ if it is nonnegative and $-$ if it is negative. If 
		one of the $v_j$, say $v_1$, is of the form $(+,+)$, then we may choose 
		$\beta=\alpha+v_1$. Similarly, if $v_1$ is of the form $(-,-)$, then we 
		choose $\beta=\alpha-v_1$. It remains therefore only to consider the 
		two combinations $(+,-)$, $(+,-)$ and $(+,-)$, $(-,+)$, where we in 
		either case may assume that all plus signs correspond to positive 
		entries.
		
		In the first case, at least one of the two vectors $v_1-v_2$ and 
		$v_2-v_1$ will be $\alpha$-admissible. If, say, $v_1-v_2$ is 
		$\alpha$-admissible, then we observe that $\ind(v_1, 
		v_1-v_2)>\ind(v_1,v_2)$. In the second case, we have plainly $\ind(v_1, 
		v_1+v_2)>\ind(v_1,v_2)$.
		
		After this initial iteration, we have two new linearly independent 
		$\alpha$-admissible vectors with a larger negativity index. We are done 
		if one of the vectors is of the form $(+,+)$ or $(-,-)$. Otherwise we 
		repeat the iteration. Since the negativity index increases in each 
		step, this iteration will eventually terminate with one of the vectors 
		being of the form $(+,+)$ or $(-,-)$. This vector is necessarily 
		nonzero because the two vectors are linearly independent. 
	\end{proof}
	
	We turn to the final and most difficult case $d=k=3$. By 
	Theorem~\ref{thm:shapirogeometric} and Example~\ref{ex:d3k3}, it remains to 
	show that $E_2^\infty(T)=\Lambda(T)\cap \mathbb N_0^3$ when $T$ is a subset 
	of $\mathbb N^3_0$ with $\dim (T)=3$. Again appealing to 
	Lemma~\ref{lem:pluscompletion}, we see that this follows from the following 
	lemma.
	\begin{lemma}\label{lem:positive} 
		Let $T$ be a set of four affinely independent points in 
		$\mathbb{N}^3_0$. Then for every $\alpha$ in $T$ there exists a point 
		$\beta$ in $E_2^\infty(T)\setminus\{\alpha\}$ such that $\beta-\alpha$ 
		is in $\mathbb{N}^3$. 
	\end{lemma}
	\begin{proof}
		We begin as in the preceding case by noting that it suffices to find a 
		$\beta$ in $E_2^{\infty}(T) \setminus T$ with $\beta-\alpha$ in 
		$\mathbb N_0^3$. Should only one of the entries be $0$, we may make a 
		similar adjustment as in the proof of Lemma~\ref{lem:positive2}. Should 
		two of the entries be $0$, then we make the following elaboration of 
		this argument. Since $\dim(T)=3$, we can find two points $\tau_1$ and 
		$\tau_2$ in $T$ such that $\beta-\alpha$ is not in the plane spanned by 
		$\tau_1-\alpha$ and $\tau_2-\alpha$. We now claim that we may replace 
		$\beta$ by $\alpha+m(\beta-\alpha)+k(\tau_1-\alpha)+\ell 
		(\tau_2-\alpha)$ for a large positive integer $m$ and suitable integers 
		$k$ and $\ell$. We see that this can be achieved by applying 
		Lemma~\ref{lem:positive2} to the two entries of $\alpha, \tau_1, 
		\tau_2$ for which $\beta-\alpha$ is $0$. 
		
		We now turn to the sequence of $2$-extensions needed to reach the 
		desired point $\beta$, starting from any of the points $\alpha$ in $T$. 
		Our plan is to act in the same way as was done in the case $d=2$. Hence 
		we wish to prove that there exists a $k$ such that at least one of the 
		two assertions is true: 
		\begin{itemize}
			\item[(i)] There is a nonzero vector in $E^k_2(T)-\alpha$ with only 
			nonnegative entries. 
			\item[(ii)] There are three linearly independent vectors $v_1'$, 
			$v_2'$, $v_3'$ in $E^k_2(T)-\alpha$ such that
			\[\ind(v_1',v_2',v_3')<\ind(v_1,v_2,v_3).\]
		\end{itemize}
		We are done if (i) is true, and if (ii) holds, then the argument can be 
		iterated, starting with the $\alpha$-admissible vectors $v_1'$, $v_2'$, 
		$v_3'$ in place of $v_1$, $v_2$, $v_3$. Our algorithm will be such that 
		the initial assumption that $v_1$, $v_2$, $v_3$ be linearly independent 
		will automatically guarantee that $v_1'$, $v_2'$, $v_3'$ be linearly 
		independent. It is clear that this iteration will eventually produce a 
		nonzero vector in $E^k_2(T)-\alpha$ for some $k$.
		
		We begin by identifying combinations of signs of the entries of the 
		vectors $v_j$ that immediately lead to the desired $\beta$. To this 
		end, we represent again an entry symbolically by $+$ if it is 
		nonnegative and $-$ if it is negative. If one of the $v_j$, say $v_1$, 
		is of the form $(+,+,+)$, then we may choose $\beta=\alpha+v_1$. 
		Similarly, if $v_1$ is of the form $(-,-,-)$, then we choose 
		$\beta=\alpha-v_1$. We assume therefore that neither of the vectors 
		$v_j$ are of these kinds. 
		
		We have found it convenient to group our treatment of the remaining 
		nontrivial combinations of signs into seven cases. The first three 
		cases deal only with combinations of two vectors; we are then either 
		able to reach the desired increase of the negativity index or we are led 
		to consider a combination of signs of three vectors which is then 
		treated later.
		
		\subsubsection*{Case 1: $(+,-,+)$ and $(+,+,-)$} Suppose that $v_1$ is 
		of the form $(+,-,+)$ and $v_2$ is of the form $(+, +,-)$. Then plainly 
		$v_1+v_2$ is again $\alpha$-admissible. We may assume that at least one 
		of the two entries $v_{1,3}$ and $v_{2,2}$ is nonzero. Indeed, if this 
		were not the case, then at least one of the vectors $v_1-v_2$ and 
		$v_2-v_1$ would be $\alpha$-admissible, and then we could replace $v_1$ 
		by $v_1-v_2$ or $v_2$ by $v_2-v_1$ to force one of the entries in 
		question to be nonzero. If, say $v_{1,3}\neq 0$, then we will increase 
		the minimal value of the third entry if we replace $v_2$ by $v_1+v_2$. 
		If the new vector $v_1+v_2$ is of the form $(+,+,+)$, then we are 
		plainly done; if it is of the form $(+,+,-)$, then we may iterate the 
		same argument, now applying it to the two vectors $v_1$ and $v_1+v_2$. 
		If it is of the form $(+,-,+)$ or $(+,-,-)$, then we bring in $v_3$ and 
		note that we have increased $\ind(\{v_1,v_2,v_3\})$ unless $v_{3,3}<0$. 
		If $v_3$ is either of the form $(+,+,-)$ or $(-,+,-)$, we would then 
		achieve the desired increase of the negativity index by replacing $v_3$ 
		by $v_3+v_1$. If both $v_1+v_2$ and $v_3$ are of the form $(+,-,-)$, 
		then we obtain the desired increase by replacing $v_3$ by one of the 
		vectors $\pm(v_3-v_1-v_2)$. The only remaining case to be considered is 
		therefore that $v_1+v_2$ is of the form $(+,-,+)$ and $v_3$ is of the 
		form $(+,-,-)$. We will treat it as Case 6 below.
		
		\subsubsection*{Case 2: $(+,-,+)$ and $(-,+,-)$} Suppose next that 
		$v_1$ is of the form $(+,-,+)$ and $v_2$ is of the form $(-,+,-)$. We 
		have a nontrivial situation if both $v_{2,2}>0$ and at least one of the 
		two entries $v_{1,1}$ and $v_{1,3}$ is positive. Assume, say, that 
		$v_{1,3}>0$. We may assume that $v_{1}+v_2$ is not of the form 
		$(-,+,-)$, since otherwise $v_1$ and $v_1+v_2$ are two vectors of the 
		same form as the initial ones, and thus we may iterate the argument. 
		Now if $v_1+v_2$ is of one of the forms $(-,+,+)$ or $(+,+,-)$, then we 
		are back to the preceding case and may proceed accordingly. The 
		remaining possibilities are that $v_1+v_2$ is of one of the forms 
		$(-,-,+)$, $(+,-,-)$, $(+,-,+)$. We bring again in $v_3$ which must be 
		of the form $(+,-,-)$ unless we already achieved the desired increase of 
		the negativity index. Should $v_1+v_2$ be of the form $(+,-,-)$, then we 
		may $v_3$ replace by one of the vectors $\pm(v_3-v_1-v_2)$. The two 
		remaining cases will be dealt with as respectively Case 6 and Case 7 
		below. 
		
		\subsubsection*{Case 3: $(+,-,-)$ and $(+,-,-)$} If we have a 
		combination with $v_1$ of the form $(+,-,-)$ and $v_2$ of the form 
		$(+,-,-)$, then it is plain that at least one of the two vectors 
		$v_1-v_2$ and $v_2-v_1$ is $\alpha$-admissible. Should the new vector 
		$u$ be of the same form, we may iterate the argument with $v_2$ 
		replaced by $u$. Then after a finite number of iterations, we either 
		reach a vector of the desired form $(+,+,+)$ or we end up with a 
		combination like $(+,-,-)$ and $(+,-,+)$. This situation is covered by 
		Case 6 and Case 7 below.
		
		Up to inessential permutations, it now remains to check the following 
		possible combinations of signs: 
		\begin{equation}\label{eq:case4} 
			\begin{pmatrix}
				v_1 \\
				v_2 \\
				v_3 
			\end{pmatrix}
			= 
			\begin{pmatrix}
				- & - & + \\
				+ & - & - \\
				- & + & - 
			\end{pmatrix}
			; 
		\end{equation}
		\begin{equation}
			\begin{pmatrix}\label{eq:pluses} 
				v_1 \\
				v_2 \\
				v_3 
			\end{pmatrix}
			= 
			\begin{pmatrix}
				+ & - & + \\
				+ & - & + \\
				+ & - & + 
			\end{pmatrix}
			; 
		\end{equation}
		\begin{equation}\label{eq:zero} 
			\begin{pmatrix}
				v_1 \\
				v_2 \\
				v_3 
			\end{pmatrix}
			= 
			\begin{pmatrix}
				- & - & + \\
				+ & - & + \\
				+ & - & + 
			\end{pmatrix}
			; 
		\end{equation}
		\begin{equation}\label{eq:nonzero} 
			\begin{pmatrix}
				v_1 \\
				v_2 \\
				v_3 
			\end{pmatrix}
			= 
			\begin{pmatrix}
				- & - & + \\
				+ & - & - \\
				+ & - & + 
			\end{pmatrix}
			. 
		\end{equation}
		
		\subsubsection*{Case 4: The combination \eqref{eq:case4}} If a 
		nonnegative entry, in absolute value, is less than or equal to the two 
		other entries in the same column, then we may replace the corresponding 
		vector by $v$ by $-v$ without changing the negativity index of the 
		three vectors $v_1$, $v_2$, $v_3$. That leads us to \eqref{eq:nonzero} 
		(see below), that will be treated later. Otherwise, $v_1+v_2+v_3$ is 
		$\alpha$-admissible, and we increase the negativity index if we replace 
		the three vectors $v_1$, $v_2$, $v_3$ by $v_1$, $v_2$, $v_1+v_2+v_3$. 
		Notice that $\alpha+v_1+v_2+v_3$ is indeed in 
		$E_2(\{\alpha,\alpha+v_1,\alpha+v_2, \alpha+v_3\})$ because
		\[\alpha+v_1+v_2+v_3=(\alpha+v_1)+(v_2-v_1)+(v_3-v_1)-(-v_1).\]
		
		\subsubsection*{Case 5: The combination \eqref{eq:pluses}} We may 
		assume that $2v_{i,2}<\min_{j} v_{j,2}$ since otherwise we may add $v_i$ 
		to $v_i$ as many times as needed to achieve this. These operations will 
		not change the negativity index of the three vectors. Suppose that the 
		largest value in the first column is $v_{1,1}$. In this case, if 
		$v_{1,3}\geq v_{i,3}$ for $i\neq 1$, we then get a larger value in the 
		second entry by replacing $v_1$ by $v_1-v_i$. Iterating this argument, 
		we see that we may assume that $v_{1,3}<v_{i,3}$ for $i\neq 1$. Assume 
		similarly that the maximum in the third column is $v_{2,3}$ and that 
		$v_{2,1}<v_{i,1}$ for $i\neq 2$. Assume first that 
		$v_{1,2}=v_{2,2}=v_{3,2}$. Then $v_4:=v_{1}+v_2-v_3$ has the same 
		second entry but with $v_{2,1} \leq v_{4,1}<v_{1,1}$ and $v_{1,3} \leq 
		v_{4,3}<v_{2,3}$. If we have equality in any of the two inequalities to 
		the left, then $v_2-v_4$ or $v_1-v_4$ will be of the form $(+,+,+)$ so 
		that we have reached the desired $\beta$. Otherwise we replace $v_1$ by 
		$v_4$, which implies that we have decreased the first entry and 
		increased the third entry of the first vector. Iterating, we see that 
		we will then reach our desired $\beta$ in a finite number steps. 
		
		Hence we may assume in what follows that $v_{i,2}$ is not the same for 
		all $i=1,2,3$. If now $v_{3,2}< \max(v_{1,2},v_{2,2})$, then 
		$v_4:=v_{1}+v_2-v_3$ is $\alpha$-admissible, and its second entry is $> 
		\min(v_{2,2}, v_{1,2})$. Hence, if say $v_{1,2}=\max(v_{1,2}, 
		v_{2,2})$, then the negativity index of the vectors $v_1$, $v_3$, 
		$v_{4}$ is strictly larger than that of $v_1$, $v_2$, $v_3$. If 
		$v_{3,2}= \max(v_{1,2},v_{2,2})$, then still $v_4$ is 
		$\alpha$-admissible and $v_{4,2}=\min(v_{2,2}, v_{1,2})$. If, say, 
		again $v_{1,2}=\max(v_{1,2}, v_{2,2})$, then we may replace $v_2$ by 
		$v_4$ and iterate the arguments just given. To simplify the notation, 
		let $v_2$ denote the replacement found for $v_2$ at any stage of the 
		iteration. If eventually $v_{2,3}=v_{3,1}$, then the iteration will 
		terminate because $v_1-v_2$ will be of the form $(+,+,+)$. Otherwise, 
		since $v_{2,3}$ will be strictly decreasing as long as 
		$v_{2,3}>v_{3,3}$, we see that $v_{2,3}$ will eventually be smaller than 
		or equal to $v_{3,3}$. If also $v_{3,1}\geq v_{2,1}$, then $v_3-v_2$ 
		will now be of the form $(+,+,+)$. Should this not be the case, then we 
		may interchange the roles of $v_2$ and $v_3$ and eventually obtain that 
		$v_{3,2}< \max(v_{1,2},v_{2,2})$, as in the preceding case. 
		
		It remains to consider the case when both $v_{3,2}> v_{1,2}$ and 
		$v_{3,2} > v_{2,2}$. If both $v_{1,1}\geq 2v_{3,1}$ and $v_{2,3}\geq 
		2v_{3,3}$, then we see that $v_4:=v_1+v_2-2v_3$ will be 
		$\alpha$-admissible. Indeed, since by assumption 
		$2v_{3,2}<\min(v_{1,2},v_{2,2})$, wee see that 
		$v_{4,2}>\max(v_{1,2},v_{2,2})$. If, say, $v_{1,2}<v_{2,2}$, then the 
		negativity index of $v_4$, $v_2$, $v_3$ will exceed that of $v_1$, 
		$v_2$, $v_3$. Otherwise, if $v_{1,2}=v_{2,2}$, then we replace $v_1$ by 
		$v_4$ and return to the starting point of the argument, noting that our 
		gain in this first step is a strict increase of the second entry of the 
		vector $v_1$.
		
		Finally, if we have $v_{1,1}< 2v_{3,1}$, then the vector 
		$v_4:=2v_3-v_1$ will satisfy $v_{4,2}>\min(v_{1,2}, v_{2,2})$, and if we 
		have $v_{2,3}<2v_{3,3}$, then $v_4:=2v_3-v_2$ will satisfy the same 
		inequality. We may in either case proceed in exactly the same way as 
		when both $v_{1,1}\geq 2v_{3,1}$ and $v_{2,3}\geq 2v_{3,3}$.
		
		\subsubsection*{Case 6: The combination \eqref{eq:zero}} We may assume 
		that the largest value in the first column is $v_{2,1}$. If 
		$v_{2,3}\geq v_{1,3}$, then we see that either $v_2-v_1$ is of the form 
		$(+,+,+)$ or we reach Case 5 by replacing $v_1$ by the 
		$\alpha$-admissible vector $v_2-v_1$. So we may assume that $v_{2,3}< 
		v_{1,3}$. If instead $v_{3,3}\geq v_{1,3}$, then in a similar fashion 
		$v_3-v_1$ is of the form $(+,+,+)$ or we may replace $v_1$ by $v_3-v_1$ 
		to once again return to Case 5. We may therefore assume that both 
		$v_{1,3}>v_{2,3}$ and $v_{1,3}>v_{3,3}$. If now $v_{2,3}\geq v_{3,3}$, 
		then $v_2-v_3$ is $\alpha$-admissible; if it is not of the form 
		$(+,+,+)$, then we may replace $v_2$ by $v_2-v_3$ and repeat the 
		reasoning just made. This iteration will either produce a vector of the 
		form $(+,+,+)$ or a situation in which $v_{2,3}$ is maximal in the 
		first column and $v_{2,3}<v_{3,3}<v_{1,3}$. From this point on, we may 
		follow word for word the reasoning in the preceding Case 5, now with the 
		roles of $v_1$ and $v_2$ interchanged. 
		
		\subsubsection*{Case 7: The combination \eqref{eq:nonzero}} Assume 
		first that $v_{3,3}$ is the largest value in the third column. Then we 
		may replace $v_1$ by $v_3-v_1$ so that either $v_3-v_1$ is of the form 
		$(+,+,+)$ or we have reduced our problem to the preceding Case 6. 
		Similarly, if $v_{3,1}$ is the largest value in the first column, then 
		either $v_3-v_2$ is of the form $(+,+,+)$ or we reduce our problem to 
		Case 6 by replacing $v_2$ by $v_3-v_2$. We consider the final 
		possibility that $v_{2,1}$ is maximal in the first column and $v_{1,3}$ 
		is maximal in the third column. Hence we may assume that 
		$v_{2,1}>v_{3,1}$ and $v_{1,3}>v_{3,3}$. We have now plainly that 
		$v_{1,1}$ is minimal in the first column and that $v_{2,1}$ is minimal 
		in the third column. This allows us to follow word for word the 
		reasoning in Case 5, again with the roles of $v_1$ and $v_2$ 
		interchanged. 
	\end{proof}

	\section{\texorpdfstring{$H^p(\mathbb{T}^\infty)$}{Hp(Tinfty)} and 
		applications to Hardy spaces of Dirichlet series} \label{sec:bohr}
	
	\subsection{Hardy spaces on the infinite-dimensional torus} 
	\label{subsec:dinfty} Since $\mathbb T^\infty$ is a compact abelian group, 
	Theorem~\ref{thm:Tdproj} remains true for $d=\infty$ if we use the Haar 
	measure $m_{\infty}$ of $\mathbb T^\infty$ to define $L^p(\mathbb 
T^\infty)$. To 
	this end, we may as before either rely on combining the results of Rudin 
	\cite{Rudin} and And\^{o} \cite{Ando66} as indicated above or simply repeat 
	our proof in Section~\ref{subsec:Tdproof} word for word. It is also plain 
	that Theorem~\ref{thm:TdprojHp} and Theorem~\ref{thm:shapirogeometric} 
	remain valid when we set $d=\infty$. In the latter case, it should be noted 
	that all subsets $\Gamma$ of $\mathbb{N}_0^{(\infty)}$ will consist of 
	points with finitely many nonzero entries and that $E_n(\Gamma)$ can be 
	defined in exactly the same way as in the finite-dimensional case. We refer 
	to \cite{BBSS19} and to \cite[Ch.~6]{QQ21} for further information about 
	the spaces $H^p(\mathbb T^\infty)$. 
	\begin{proof}
		[Proof of Theorem~\ref{thm:curlyop}] We will apply 
		Theorem~\ref{thm:shapirogeometric} with Example~\ref{ex:d3} for $n=1$ 
		and Example~\ref{ex:d4} for $n\geq2$. We go through the details only in 
		the latter case, since the former is completely analogous. Consider
		\[\Gamma_n := 
		\left\{(n,1,0,1),(n+1,0,1,0),(0,0,n+1,0),(0,0,0,n+1)\right\}\]
		for $n\geq2$. By Theorem~\ref{thm:shapirogeometric}, we know that 
		$\Gamma_n$ is a contractive projection set for $H^p(\mathbb{T}^4)$ if 
		and only if $p=2,4,\ldots, 2(n+1)$. Decompose $\mathbb{T}^\infty$ into 
		a infinite cartesian product of four-dimensional tori,
		\[\mathbb{T}^\infty = \mathbb{T}_1^4 \times \mathbb{T}_2^4 \times 
		\mathbb{T}_3^4 \times \cdots,\]
		where $\mathbb{T}_j^4$ contains the variables 
		$z_{4j-3},z_{4j-2},z_{4j-1}$, and $z_{4j}$.
		
		For $m\geq1$, let $T_{m,n}$ be the operator defined by letting the 
		projection $P_{\Gamma_n}$ act on each of the $m$ four-dimensional tori 
		$\mathbb{T}_{(m-1)m/2+1}^4,\,\ldots,\, \mathbb{T}_{m(m+1)/2}^4$ 
		independently. Clearly, 
		\begin{equation}\label{eq:norm4} 
			\|T_{m,n}\|_{H^p(\mathbb{T}^\infty)\to H^p(\mathbb{T}^\infty)} = 
			\|P_{\Gamma_n}\|_{H^p(\mathbb{T}^4)\to H^p(\mathbb{T}^4)}^m. 
		\end{equation}
		Define the operator $T_n$ by 
		\begin{equation}\label{eq:Tmn} 
			T_n f = \sum_{m=1}^\infty \frac{T_{m,n} f}{m^2}. 
		\end{equation}
		The operator \eqref{eq:Tmn} is well-defined for $f$ in 
		$H^p(\mathbb{T}^d)$ for every finite $d$, since in this case $T_{m,n} f 
		= 0$ for every sufficiently large $m$. From this we conclude that $T_n$ 
		is densely defined on $H^p(\mathbb{T}^\infty)$ (see 
		e.g.~\cite[Thm.~2.1]{BBSS19}).
		
		We first consider the case when $p=2k$ for some integer $1 \leq k \leq 
		n+1$. Since $\|P_{\Gamma_n}\|_{H^p(\mathbb{T}^4)\to 
			H^p(\mathbb{T}^4)}=1$ by Theorem~\ref{thm:shapirogeometric}, we get 
		from \eqref{eq:norm4} and the triangle inequality that
		\[\|T_n f\|_p \leq \frac{\pi^2}{6} \|f\|_p,\]
		so the operator \eqref{eq:Tmn} is well-defined on 
		$H^p(\mathbb{T}^\infty)$ with norm at most $\pi^2/6$.
		
		Consider next the case when $1 \leq p \leq \infty$, $p\neq 2n$, for $1 
		\leq k \leq n+1$. Since $P_{\Gamma_n}$ is not a contraction on 
		$H^p(\mathbb{T}^4)$ we have
		\[\|T_{m,n}\|_{H^p(\mathbb{T}^\infty)\to H^p(\mathbb{T}^\infty)} = 
		(1+\delta_p)^m\]
		for some $\delta_p>0$. Since each $T_{m,n}$ acts on separate variables, 
		we get from \eqref{eq:Tmn} that
		\[\|T_n\|_{H^p(\mathbb{T}^\infty)\to H^p(\mathbb{T}^\infty)} \geq 
		\frac{(1+\delta_p)^m}{m^2}\]
		for every positive integer $m$ and, consequently, $T_n$ is unbounded on 
		$H^p(\mathbb{T}^\infty)$. 
	\end{proof}
	\begin{problem}\label{prob:operator} 
		Is there a linear operator $T$ that is densely defined on 
		$H^p(\mathbb{T}^\infty)$ for $1 \leq p \leq \infty$ and extends to a 
		bounded operator on $H^p(\mathbb{T}^\infty)$ if and only if $p$ is an 
		even integer or $p=\infty$? 
	\end{problem}
	
	\subsection{Hardy spaces of Dirichlet series} \label{subsec:curly} For $1 
	\leq p < \infty$, the Hardy space of Dirichlet series $\mathscr{H}^p$ can 
	be defined as the closure of the set of Dirichlet polynomials $f(s) = 
	\sum_{n=1}^N a_n n^{-s}$ in the Besicovitch norm
	\[\|f\|_{\mathscr{H}^p}^p = \lim_{T\to \infty} \frac{1}{2T} \int_{-T}^T 
	|f(it)|^p\,dt.\]
	The endpoint case $\mathscr{H}^\infty$ is comprised of somewhere convergent 
	Dirichlet series that may be analytically continued to bounded analytic 
	functions in the half-plane $\mre{s}>0$, and we set
	\[\|f\|_{\mathscr{H}^\infty} := \sup_{\mre{s}>0}|f(s)|.\]
	Let $(p_j)_{j\geq1}$ denote the increasing sequence of prime numbers. By 
	the fundamental theorem of arithmetic, every positive rational number is 
	uniquely represented as
	\[q = \prod_{j=1}^d p_j^{\alpha_j} \qquad \longleftrightarrow \qquad 
	\alpha(q) = (\alpha_1,\alpha_2,\ldots,\alpha_d,0,0,0,\ldots).\]
	This representations associates to each $q$ in $\mathbb{Q}_+$ a unique 
	multi-index $\alpha$ in $\mathbb{Z}^{(\infty)}$. It follows that the groups 
	$(\mathbb{Z}^{(\infty)},+)$ and $(\mathbb{Q}_+,\times)$ are isomorphic. 
	Note that the subset $\mathbb{N}_0^{(\infty)}$ of $\mathbb{Z}^{(\infty)}$ 
	is identified with the subset $\mathbb{N}$ of $\mathbb{Q}_+$. 
	
	The Bohr correspondence
	\[\mathscr{B}f(z) := \sum_{n=1}^\infty a_n z^{\alpha(n)}\]
	defines an isometric isomorphism from $\mathscr{H}^p$ to 
	$H^p(\mathbb{T}^\infty)$. In the range $1 \leq p < \infty$ this can be 
	established either by the ergodic theorem (see \cite[Sec.~2]{Bayart02}) or 
	by a simple argument using the Weierstrass approximation theorem (see 
	\cite[Sec.~3]{SS09}). For $p=\infty$ we can prove the isometric isomorphism 
	by taking the limit $p\to\infty$ and using Bohr's theorem \cite{Bohr13}, 
	which guarantees that Dirichlet series in $\mathscr{H}^\infty$ are 
	uniformly convergent in the half-plane $\mre{s}\geq \delta$ for every 
	$\delta>0$.
	
	A set $\Gamma \subseteq \mathbb{N}$ is called a contractive projection set 
	for $\mathscr{H}^p$ if the projection $P_\Gamma f(s) := \sum_{n \in \Gamma} 
	a_n n^{-s}$ is a contraction on $\mathscr{H}^p$. In view of the discussion 
	in Section~\ref{subsec:dinfty}, we may then translate 
	Theorem~\ref{thm:TdprojHp} for the special case $d=\infty$ into a 
	corresponding assertion about $\mathscr{H}^p$. We refrain from carrying out 
	the details, noting only that we will then be dealing with restrictions of 
	sets and cosets in $\mathbb{Q}_+$ to $\mathbb{N}$.
	
	In this context, computations with contractive project sets for 
	$\mathscr{H}^p$ will frequently involve arithmetic functions. A useful 
	example, alluded to in the introduction and employed in \cite{BS16}, is the 
	set
	\[\Gamma_m := \left\{n \in \mathbb{N}\,:\,\Omega(n)=m\right\},\]
	where $\Omega(n)$ denotes the number of prime factors of $n$ (counting 
	multiplicities). As in Example~\ref{ex:mhom}, the formula for the 
	projection is
	\[P_{\Gamma_m}f(s) = \int_{\mathbb{T}} \left(\sum_{n=1}^\infty a_n 
	w^{\Omega(n)} n^{-s} \right)\,w^{-m}\,dm_1(w)\]
	for $f(s) = \sum_{n\geq1} a_n n^{-s}$ in $\mathscr{H}^p$.
	
	We may reformulate Theorem~\ref{thm:curlyop} to obtain the following result.
	\begin{corollary}\label{cor:curlyop} 
		Fix an integer $n\geq1$. There is a linear operator $T_n$ which is 
		densely defined on $\mathscr{H}^p$ for every $1 \leq p \leq \infty$, 
		and which does not extend to a bounded operator on $\mathscr{H}^p$ 
		unless $p=2,4,\ldots,2(n+1)$.
	\end{corollary}
	
	Perhaps the most important open problem in the study of the spaces 
	$\mathscr{H}^p$ is the local embedding problem \cite[Prob.~2.1]{SS16}, 
	which asks whether there is a constant $C_p>0$ such that 
	\begin{equation}\label{eq:localemb} 
		\int_{-1}^1 |f(1/2+it)|^p\,dt \leq C_p \|f\|_{\mathscr{H}^p}^p 
	\end{equation}
	for every $f$ in $\mathscr{H}^p$. The answer to the embedding problem is 
	known to be positive if $p$ is an even integer (see e.g. 
	\cite[Sec.~3]{SS09}) and negative if $p < 2$ as a corollary to a recent 
	result of Harper \cite{Harper20}. 
	
	The work of Bayart and Masty{\l}o \cite{BM19}, which we discussed in the 
	introduction, demonstrates that the standard interpolation techniques 
	cannot be employed to extend the positive conclusion for the embedding 
	problem from even integers $p$ to general $p>2$. Based on this and the 
	analogy with the Hardy--Littlewood majorant principle (elucidated in 
	\cite[Sec.~7.3]{Montgomery94}), it is conjectured in \cite[p.~274]{QQ21} 
	that the local embedding should hold only for even integers $p$ (and, 
	trivially, for $p=\infty$).
	
	The local embedding problem can be restated in terms of the boundedness of 
	a densely defined linear operator. Consider the composition operator 
	defined on $\mathscr{H}^p$ by $\mathscr{C}_\varphi f := f \circ \varphi$ for
	\[\varphi(s) := \frac{1}{2} + \frac{1-2^{-s}}{1+2^{-s}}.\]
	The operator is well-defined and bounded on $\mathscr{H}^\infty$ by Bohr's 
	theorem, so it is densely defined on $\mathscr{H}^p$ for every $1 \leq p 
	\leq \infty$. We know from \cite[Thm.~3]{BB19} that $\mathscr{C}_\varphi$ 
	is bounded on $\mathscr{H}^p$ if and only if the local embedding 
	\eqref{eq:localemb} holds. We note in passing that if $p$ is an even 
	integer or $p=\infty$, then actually 
	$\|\mathscr{C}_\varphi\|_{\mathscr{H}^p\to\mathscr{H}^p} = 2^{1/p}$ by 
	results in \cite{Brevig17} and \cite[Sec.~8.11]{QQ21}. 
	
	The operator $\mathscr{C}_\varphi$ is unbounded for $p<2$ by Harper's 
	result. Hence, if the local embedding \eqref{eq:localemb} holds for $p>2$ 
	only if $p$ is an even integer, then $\mathscr{C}_\varphi$ is the operator 
	enquired after in Problem~\ref{prob:operator}. In this context, 
	Corollary~\ref{cor:curlyop} does not render it implausible that the local 
	embedding may hold only for even integers (and $p=\infty$).
	
	\bibliographystyle{amsplain} 
	\bibliography{cpsets} 
\end{document}